\RequirePackage{fix-cm}
\documentclass[smallextended]{svjour3}       
\smartqed  
\usepackage{hyperref}
\hypersetup{colorlinks=true,linkcolor=blue,citecolor=blue}
\usepackage{times}
\usepackage[algo2e,vlined,ruled]{algorithm2e}
\usepackage{graphicx,graphics,epsf,epstopdf,subfigure}
\usepackage{comment}
\usepackage{amsmath,amsxtra,amsfonts,amscd,amssymb,bm}
\allowdisplaybreaks 
\usepackage{supertabular,multirow}
\usepackage{booktabs}
\usepackage{scrextend} 

\usepackage{tikz,array}
\usetikzlibrary{shapes.geometric}
\usetikzlibrary{shapes.arrows}
\usetikzlibrary{calc}
\usetikzlibrary{intersections}

\newcommand{\R}{\mathbb{R}}

\newcommand{\Rmn}{\mathbb{R}^{m\times n}}

\newcommand{\ff}{_{\mathrm{F}}}  
\newcommand{\fs}{^2_{\mathrm{F}}}
\newcommand{\half}{\frac{1}{2}}
\newcommand{\st}{\mathrm{s.\,t.}\,\,} 
\newcommand{\zz}{^{\top}} 

\newcommand{\manifold}{{\cal M}}
\newcommand{\Mk}{{\cal M}_{k}}
\newcommand{\Ms}{{\cal M}_{s}}
\newcommand{\Mlek}{{\cal M}_{\le k}}

\newcommand{\stiefel}{\mathrm{St}}

\newcommand{\TX}{{\mathrm{T}_{X}}}

\newcommand{\projec}{\mathcal{P}}

\newcommand{\NXleks}{\dkh{\TX\Ms}_{\le (k-s)}^{\perp}}

\DeclareMathOperator*{\argmin}{arg\,min}
\DeclareMathOperator*{\argmax}{arg\,max}

\DeclareMathOperator*{\diag}{diag}

\DeclareMathOperator*{\rank}{rank}

\DeclareMathOperator*{\tr}{tr}
\DeclareMathOperator*{\po}{P_{\Omega}}

\newcommand{\abs}[1]{\left|#1\right|}
\newcommand{\dkh}[1]{\left(#1\right)}
\newcommand{\fkh}[1]{\left[#1\right]}
\newcommand{\hkh}[1]{\left\{#1\right\}}
\newcommand{\jkh}[1]{\left\langle#1\right\rangle}
\newcommand{\norm}[1]{\left\|#1\right\|}

\newcommand{\rgrads}[1]{\mathrm{grad}_s f(#1)}

\graphicspath{ {./images/} }
\definecolor{Gray}{rgb}{0.5,0.5,0.5}
\definecolor{myred}{rgb}{0.7,0,0.1}
\definecolor{mygreen}{rgb}{0,0.6,0.2}
\definecolor{myblue}{rgb}{0.5,0,1}

\begin{document}

\title{A Riemannian rank-adaptive method for low-rank matrix completion\thanks{This work was supported by the Fonds de la Recherche Scientifique -- FNRS and the Fonds Wetenschappelijk Onderzoek -- Vlaanderen under EOS Project no. 30468160.}
}


\author{Bin Gao         \and
        P.-A. Absil 
}

\authorrunning{B. Gao, P.-A. Absil} 

\institute{Bin Gao and P.-A. Absil \at
              ICTEAM Institute, UCLouvain, 1348 Louvain-la-Neuve, Belgium \\
              \email{gaobin@lsec.cc.ac.cn; pa.absil@uclouvain.be}           
}

\date{Received: date / Accepted: date}

\maketitle

\begin{abstract}
The low-rank matrix completion problem can be solved by Riemannian optimization on a fixed-rank manifold. However, a drawback of the known approaches is that the rank parameter has to be fixed a priori. In this paper, we consider the optimization problem on the set of bounded-rank matrices. We propose a Riemannian rank-adaptive method, which consists of fixed-rank optimization, rank increase step and rank reduction step.  We explore its performance applied to the low-rank matrix completion problem. Numerical experiments on synthetic and real-world datasets illustrate that the proposed rank-adaptive method compares favorably with  state-of-the-art algorithms. In addition, it shows that one can incorporate each aspect of this rank-adaptive framework separately into existing algorithms for the purpose of improving performance.
\end{abstract}

\section{Introduction}\label{sec:intro}

The low-rank matrix completion problems have been extensively studied in recent years; see the survey~\cite{nguyen2019low}. The matrix completion model based on a Frobenius norm minimization over the manifold of fixed-rank matrices is formulated as follows,
\begin{equation}\label{prob:fixed-rank}
\begin{aligned}
\min ~&~ f(X):=\half\norm{\po(X)-\po(A)}\fs\\
\st & X\in\Mk:=\{X\in\Rmn: \rank(X)=k\},
\end{aligned}
\end{equation}
where $A\in\Rmn$ is a data matrix only known on a subset $\Omega\subset \{1,\dots,m\}\times\{1,\dots,n\}$, $k\le\min(m,n)$ is a given rank parameter and $\po:\Rmn\to\Rmn$ denotes the projection onto $\Omega$, i.e., $\fkh{\po(X)}_{i,j}=X_{i,j}$ if $(i,j)\in\Omega$, otherwise $\fkh{\po(X)}_{i,j}=0$. When the rank of data matrix~$A$ is known or prior estimated, this model has been shown to be effective for low-rank matrix completion; see~\cite{meyer2011linear,vandereycken2013low}. However, the choice of the rank parameter $k$ affects the speed of the methods that address~\eqref{prob:fixed-rank} as well as the quality of the returned completion. In practice, if $\rank(A)\ge k$, solving problem~\eqref{prob:fixed-rank} returns a low-rank solution on $\Mk$. When $\rank(A)<k$, it is advisable to replace the constraint of problem~\eqref{prob:fixed-rank} ($X\in\Mk$) by $X\in\Ms$ for $s<k$. Considering the inconvenience of rank estimation of $A$ in general, a rank-adaptive algorithm with a flexible choice of rank parameter~$k$ is desirable. 

We consider the following model for  low-rank matrix completion:
\begin{equation}\label{prob:bounded-rank}
\begin{aligned}
\min~ & ~f(X)\\
\st & X\in\Mlek:=\{X\in\Rmn: \rank(X)\le k\}.
\end{aligned}
\end{equation}
Several algorithms based on Riemannian optimization (e.g., see~\cite{absil2009optimization}) for this problem have been developed in~\cite{bigmatrixrecovery,Uschmajew_V:2015,schneider2015convergence}. Recently, a Riemannian rank-adaptive method for low-rank optimization has been proposed in \cite{Zhou2016riemannian}, and problem~\eqref{prob:bounded-rank} can be viewed as a specific application. This rank-adaptive algorithm mainly consists of two steps: Riemannian optimization on the fixed-rank manifold and adaptive update of the rank. When a nearly rank-deficient point is detected, the algorithm reduces the rank to save computational cost.  Alternatively, it increases the rank to gain accuracy. However, there are several parameters that users need to tune.

In this paper, we propose a new Riemannian rank-adaptive method (RRAM). In comparison with the RRAM method in~\cite{Zhou2016riemannian}, we stray from convergence analysis concerns in order to focus on the efficiency of the proposed method for low-rank matrix completion. Specifically, the contributions are as follows.
\begin{itemize}
	\item We adopt a Riemannian gradient method with non-monotone line search and Barzilai--Borwein step size to solve the optimization problem on the fixed-rank manifold.
	\item By detecting the significant gap of singular values of iterates, we propose a novel rank reduction strategy such that the fixed-rank problem can be restricted to a dominant subspace. In addition, we propose a normal correction strategy to increase the rank. Note that the existing algorithms may benefit from these rank-adaptive mechanisms to improve their numerical performance. 
	\item We demonstrate the effectiveness of the proposed method applied to low-rank matrix completion. The numerical experiments on synthetic and real-world datasets illustrate that the proposed rank-adaptive method is able to find the ground-truth rank and compares favorably with other state-of-the-art algorithms in terms of time efficiency.
\end{itemize}

The value of the new rank-adaptive method is application dependent. However, our observations provide insight on the kind of data matrices for which rank-adaptive mechanisms play a valuable role. The main purpose of this paper is thus to explore the  numerical behaviors of rank-adaptive methods on synthetic and real-world problems.


The rest of paper is organized as follows. The next section introduces related work based on rank-update mechanisms, and presents necessary ingredients of the proposed method. In section~\ref{sec:algorithm}, a~new Riemannian rank-adaptive method is proposed and its implementation details are also provided. Numerical experiments are reported in section~\ref{sec:experiments}. The conclusion is drawn in section~\ref{sec:conclusion}. 

\section{Related work and preliminaries}\label{sec:preliminaries}
In this section, we start with related work and give the preliminaries regarding the geometric aspect.
\subsection{Related work}\label{subsec:related-work}
The feasible set $\Mk$ of problem~\eqref{prob:fixed-rank} is a smooth submanifold of dimension $(m+n-k)k$ embedded in $\R^{m\times n}$; see~\cite[Example 8.14]{lee2003introduction}. 
A Riemannian conjugate gradient (RCG) method for solving problem~\eqref{prob:fixed-rank} has been proposed in~\cite{vandereycken2013low}, which efficiently assembles ingredients of RCG by employing the low-rank structure of matrices.
There has been other methods for the fixed-rank optimization including the Riemannian trust-region method (RTR) \cite{trace_penalty} and the Riemannian gradient-descent method (RGD) \cite{Zhou2016riemannian}.
Mishra et al. \cite{trace_penalty} have considered a trace norm penalty model for low-rank matrix completion, and have proposed a method that alternately performs a fixed-rank optimization and a rank-one update.

However, $\Mk$ is not closed in $\R^{m\times n}$, hence $\min_{X\in\Mk}f(X)$ may not have a solution even when $f$ is continuous and coercive; moreover, if a Riemannian optimization algorithm has a limit point of rank less than $k$, then the classical convergence results in Riemannian optimization (e.g.,~\cite{BouAbsCar2018}) do not provide guarantees about the limit point.
As~a~remedy, one can resort to the set of bounded rank matrices, i.e., $\Mlek$. 
Recently, algorithms for solving problem~\eqref{prob:bounded-rank}, combining the fixed-rank Riemannian optimization with a rank-increase update, have been introduced in~\cite{bigmatrixrecovery,Uschmajew_V:2015}. Basically, these methods increase the rank with a constant by searching along the tangent cone of $\Mlek$ and projecting onto $\Mk$ or $\Mlek$. In addition, a general projected line-search method on $\Mlek$ has been developed in~\cite{schneider2015convergence} whose convergence guarantee is based on the assumption that limit points of algorithm have rank $k$. In summary, we list all these rank-related algorithms with their corresponding  features in Table~\ref{tab:rank-related}; a detailed explanation will be given in Remark~\ref{remark:tab} (page~\pageref{remark:tab}) after we introduce the necessary geometric ingredients for Riemannian optimization.

\begin{table}[htbp]
	\centering
	\caption{Rank-related algorithms based on the geometry of the feasible set\label{tab:rank-related}}
	\begin{tabular}{cclr}
		\toprule
		\multirow{2}{*}{{\sc algorithm}} & \multirow{2}{*}{\sc fixed-rank} & \multicolumn{2}{c}{{\sc rank}}\\\cmidrule(r){3-4}
		& & {\sc increase} & {\sc reduction} \\\midrule 
		\cite{trace_penalty} & RTR & $X_+ = X +\alpha wy\zz$& -\\
		\cite{bigmatrixrecovery} & RCG & $X_+ = \projec_{\mathcal{M}_{s+l}}(X +\alpha G_{\le (s+l)}) $ & -\\
		\cite{Uschmajew_V:2015} & RCG& $X_+ = X +\alpha G_{\le (s+l)}$ & -\\
		\cite{schneider2015convergence} & -& {$X_+ = \projec_{\Mlek}(X +\alpha G_{\le k}) $} & -\\
		\cite{Zhou2016riemannian} & RGD &  $X_+ = \projec_{\mathcal{M}_{\le (s+l^*)}}(X +\alpha G_{\le (s+l^*)}) $ & $\sigma_i/\sigma_1 < \Delta$\\ 
		Ours & RBB & $X_+ = X +\alpha WDY\zz$ & ${(\sigma_i  - \sigma_{i+1})}/{\sigma_i}  \!>\! \Delta$\\
		\bottomrule
	\end{tabular}
\end{table}

\subsection{Geometry of $\Mlek$}
The geometry of $\Mlek$ has been well studied in~\cite{schneider2015convergence}. In this subsection, we introduce several Riemannian aspects that will be used in the rank-adaptive method.

Given the singular value decomposition (SVD) of fixed-rank matrices, an equivalent expression of the manifold $\Mk$ is
\begin{equation*}
\Mk = \left\{ U\Sigma V\zz:
\begin{array}{c}
U\in\stiefel(k,m), V\in\stiefel(k,n),\\
\Sigma=\diag(\sigma_1,\dots,\sigma_k) ~\mbox{with}~ \sigma_1\ge\cdots\ge\sigma_k>0
\end{array}
\right\},
\end{equation*}
where 
$$\stiefel(k,m):=\hkh{X\in\R^{m\times k}: X\zz X=I_k}$$ denotes the (compact) Stiefel manifold, and a diagonal matrix with $\{\sigma_i\}$ in its diagonal is denoted by $\diag(\sigma_1,\dots,\sigma_k)$. This expression of $\Mk$ provides a convenient way to assemble other geometric tools. For instance, the tangent space of $\Ms$ at $X\in\Ms$ is given as follows; see~\cite[Proposition 2.1]{vandereycken2013low}
\begin{equation*}
\TX\Ms = \hkh{\begin{bmatrix}
	U &U_\perp
	\end{bmatrix} \begin{bmatrix}
	\R^{s\times s} & \R^{s\times (n-s)}\\
	\R^{(m-s)\times s} & 0_{(m-s)\times (n-s)} 
	\end{bmatrix} \begin{bmatrix}
	V &V_\perp
	\end{bmatrix}\zz },
\end{equation*}
where $U_\perp\in\R^{m\times(m-s)}$ denotes a matrix such that $U\zz U_\perp = 0$ and $U_\perp\zz U^{}_\perp =I$. Moreover, the normal space of $\Ms$ at $X$ associated with the Frobenius inner product, $\jkh{X,Y}:=\tr(X\zz Y)$, has the following form
\begin{equation}\label{eq:normal-space}
\dkh{\TX\Ms}^{\perp} = \hkh{\begin{bmatrix}
	U &U_\perp
	\end{bmatrix} \begin{bmatrix}
	0_{s\times s} & 0_{s\times (n-s)}\\
	0_{(m-s)\times s} & \R^{(m-s)\times (n-s)} 
	\end{bmatrix} \begin{bmatrix}
	V &V_\perp
	\end{bmatrix}\zz }.
\end{equation}

Letting $P_U:=UU\zz$ and $P^{\perp}_U:=U^{}_\perp U_\perp\zz = I-P_U$, the orthogonal projections onto the tangent space and normal space at $X$ for $Y\in\R^{m\times n}$ are
\begin{align}\nonumber
\projec_{\TX\Ms} (Y) &= P_U Y P_V + P^{\perp}_U Y P_V + P_U Y P^{\perp}_V,\\
\projec_{\dkh{\TX\Ms}^{\perp}} (Y) &= P^{\perp}_U Y P^{\perp}_V. \label{eq:normal-proj}
\end{align}

The Riemannian gradient of $f$ at $X\in\Ms$, denoted by $\rgrads{X}$, is defined as the unique element in $\TX\Ms$ such that $\jkh{\rgrads{X}, Z}=\mathrm{D}{f}(X)[Z]$ for all $Z\in\TX\Ms$, where 
$\mathrm{D}{f}(X)$ denotes the Fr\'echet derivative of ${f}$ at $X$. It readily follows from~\cite[(3.37)]{absil2009optimization} that
\begin{equation*}
\rgrads{X} = \projec_{\TX\Ms} (\nabla f(X)),
\end{equation*}
where $\nabla {f}(X)$ denotes the Euclidean gradient of ${f}$ at $X$.

When $s<k$, the tangent cone of $\Mlek$ at $X\in\Ms$ can be expressed as the orthogonal decomposition~\cite[Theorem 3.2]{schneider2015convergence}
\begin{equation}\label{eq:tangent-lek}
\TX\Mlek = \TX\Ms  \oplus \NXleks,
\end{equation}
where {$\oplus$ denotes the direct sum and}
\begin{equation*}
\NXleks:=\hkh{N\in\dkh{\TX\Ms}^{\perp}: \rank(N)\le (k-s)}
\end{equation*}
is a subset of the normal space $\dkh{\TX\Ms}^{\perp}$. Furthermore, the projection onto the tangent cone has the form \cite[Corollary 3.3]{schneider2015convergence} 
\begin{equation*}
\projec_{\TX\Mlek}(Y) \in \argmin_{Z\in\TX\Mlek} \norm{Y-Z}\ff  = \projec_{\TX\Ms} (Y)  + \projec_{\NXleks}(Y),
\end{equation*}
where $\projec_{\NXleks}(Y)\in\NXleks$ is a best rank-($k-s$) approximation of  $\projec_{\dkh{\TX\Ms}^{\perp}} (Y)=Y-\projec_{{\TX\Ms}} (Y)$.
Note that this projection is not unique when the singular values number $k-s$ and $k-s+1$ of $\projec_{\dkh{\TX\Ms}^{\perp}}(Y)$ are equal.
For simplicity, we denote 
\begin{equation*}\label{eq:normal-approx}
N_{k-s}(X):=\projec_{\NXleks}(-\nabla f(X)),
\end{equation*}
and the projections of $-\nabla f(X)$ are denoted by
\begin{align}\label{eq:Gs_define}
G_s(X)&:=\projec_{\TX\Ms}(-\nabla f(X))=-\rgrads{X},\\
G_{\le k}(X)&:= 
\projec_{\TX\Mlek}(-\nabla f(X)) = 
G_s(X) + N_{k-s}(X). \nonumber 
\end{align}

Consequently, the optimality condition of problem~\eqref{prob:bounded-rank} can be defined as follows; see~\cite[Corollary 3.4]{schneider2015convergence}.
\begin{definition}\label{def:critical}
	$X^*\in\Mlek$ is called a \emph{critical point} of optimization problem~\eqref{prob:bounded-rank} if $$\norm{G_{\le k}(X)}\ff=0.$$
\end{definition}

\section{A Riemannian rank-adaptive method}\label{sec:algorithm}
We propose a new Riemannian rank-adaptive algorithm in this section. The Riemannian Barzilai--Borwein method with a non-monotone line search is proposed to solve the fixed-rank optimization problem. Rank increase and rank reduction strategies are developed.

\subsection{Algorithmic framework}
In view of Definition~\ref{def:critical}, a critical point $X^*\in\Mlek$ of problem~\eqref{prob:bounded-rank} satisfies 
\begin{equation}\label{eq:critical-two}
\norm{G_{\le k}(X^*)}\fs=\norm{G_s(X^*)}\fs + \norm{N_{k-s}(X^*)}\fs = 0,
\end{equation}
where the first equality follows from the orthogonal decomposition~\eqref{eq:tangent-lek} of $G_{\le k}(X^*)$. This enlightens us to develop a two-stage algorithm:
\begin{itemize}
	\item[1)] solve the optimization problem on $\Ms$, i.e.,
	\begin{equation}\label{prob:fixed-s}
	\min\limits_{X\in\Ms} f(X).
	\end{equation}
	Note that $$\norm{G_s(X)}\ff=\norm{\rgrads{X}}\ff=0$$ is the {{first-order}} optimality condition of~\eqref{prob:fixed-s};
	\item[2)] search along $N_{k-s}(X)\in\NXleks$.
\end{itemize}

\begin{figure}[htbp]
	\centering 
	\tikzset{samescale/.style={
			scale=#1,
			every node/.append style={scale=#1}
		}
	} 
	\begin{tikzpicture}[
	samescale=.9,
	auto,
	decision/.style = { diamond, aspect=2, draw=gray,
		thick, fill=gray!5, text width=5em, text badly centered,
		inner sep=1pt},
	block/.style = { rectangle, draw=gray, thick, fill=gray!5,
		text width=6em, text centered, rounded corners,
		minimum height=2em },
	line/.style = { draw, thick, ->, shorten >= 0.5pt},
	]
	
	\node [block, fill=cyan!20,
	text width=1.5em, text centered, minimum
	height=1.5em] at (-3.2,-3) (rank-increase) {};
	\node [right] at (-2.9,-3)  {fixed-rank optimization};
	
	\node [block, fill=red!20,
	text width=1.5em, text centered, minimum
	height=1.5em] at (-3.2,-3.5) (rank-increase) {};
	\node [right] at (-2.9,-3.5)  {rank increase};
	
	\node [block, fill=green!20,
	text width=1.5em, text centered, minimum
	height=1.5em] at (-3.2,-4) (rank-increase) {};
	\node [right] at (-2.9,-4)  {rank reduction};
	
	\node [block,rectangle, draw=gray, thick,dashed, fill=green!5,
	text width=22em, text centered, rounded corners, minimum
	height=6em] at (3.8,1.4) (rank-reduction) {};
	\node [gray!70] at (1.5,2.5)  {rank reduction};

	\node [block,rectangle, draw=gray, thick,dashed, fill=red!5,
	text width=22em, text centered, rounded corners, minimum
	height=11em] at (3.8,-2.1) (rank-increase) {};
	\node [gray!70] at (6.2,-3.85)  {rank increase};
	
	\node [block] at (-1.8,1.4) (initial) {initialization $X_0, \Delta$};
	
	\node [block,text width=10em,fill=green!20] at (-1.8,0) (initial-reduction) {rank reduction\\ $(\tilde{X}_{0},s) \leftarrow (X_{0},\Delta)$};
	
	\node [block,text width=10em,fill=cyan!20] at (3.8,0) (subproblem) {${X_{p}}\approx\argmin\limits_{X\in\Ms} f(X)$};
	
	\node [block,text width=10em,fill=green!20] at (3.8,1.4) (reduction) {rank reduction\\$(\tilde{X}_{p},s) \leftarrow (X_{p},\Delta)$};
	
	\node [decision, text width=7em] at (7.9,0) (decision2) {\vspace{-2 mm}\\detect large gap? \\$\hspace{-2.3mm}{{(\sigma_i  - \sigma_{i+1})}/{\sigma_i}  \!>\! \Delta }$};
	\node [right] at (7.9,1.2)  {Yes};
	\node [right] at (7.9,-1.2)  {No};
	\node [diamond, aspect=2, draw=gray,
	thick, fill=gray!5, text width=9.4em, text badly centered,
	inner sep=1pt] at (3.8,-4.8) (decision4) {\vspace{-4mm}\\increase rank?\\{$\norm{N_{k-s}}\ff > \epsilon\norm{G_s}\ff$}};
	\node [below] at (1.8,-4.8)  {No};
	\node [right] at (3.8,-3.8)  {Yes};
	
	\node [block,fill=red!20,text width=10em] at (3.8,-3.1) (increase) {set increase number $l$};

	\node [block,text width=20em, fill=red!20] at (3.8,-2.235) (increase1) {$WDY\zz \in \argmin\limits_{\rank(N)\le {l}} \norm{-\nabla f-G_s -N}\ff$};
	\node [block, text width=13em, fill=red!20] at (3.8,-1.2) (linesearch2) {line search  $(\tilde{X}_{p}, s\leftarrow s+\tilde{l})$  $\min_{\alpha> 0} f(X_p+\alpha WDY\zz)$};
	
	
	\begin{scope} [every path/.style=line,thick,shorten >= 0.5pt]
	\path (initial)    --   (initial-reduction);   
	\path (initial-reduction)      --   (subproblem);
	\path (reduction)		--		 (subproblem);
	\path (subproblem) -- (decision2);
	\path (decision2) -- (7.9,2.7) -- (3.8,2.7) -- (reduction);
	\path (decision2) -- (7.9,-4.8) -- (decision4);
	\path (decision4) -- (increase);
	\path (increase) -- (increase1); 
	\path (increase1) -- (linesearch2);
	\path (linesearch2) -- (subproblem);
	\path (decision4) -- (0.2,-4.8) -- (0.2,0);
	\end{scope}
	
	
	\end{tikzpicture}
	\caption{\label{fig:flowchart} Flowchart of the Riemannian rank-adaptive method (RRAM); see~\eqref{eq:initialization} for the initialization, Algorithm~\ref{alg:non-monotone gradient} for the fixed-rank optimization, Algorithm~\ref{alg:rank-increase} for the rank increase, and Algorithm~\ref{alg:rank-reduction} for the rank reduction, subsection~\ref{subsec:rank-increase-numerical} for the ``increase rank'' test, and subsection~\ref{subsec:rank-reduction-numerical} for the ``detect large gap'' test.}
\end{figure}
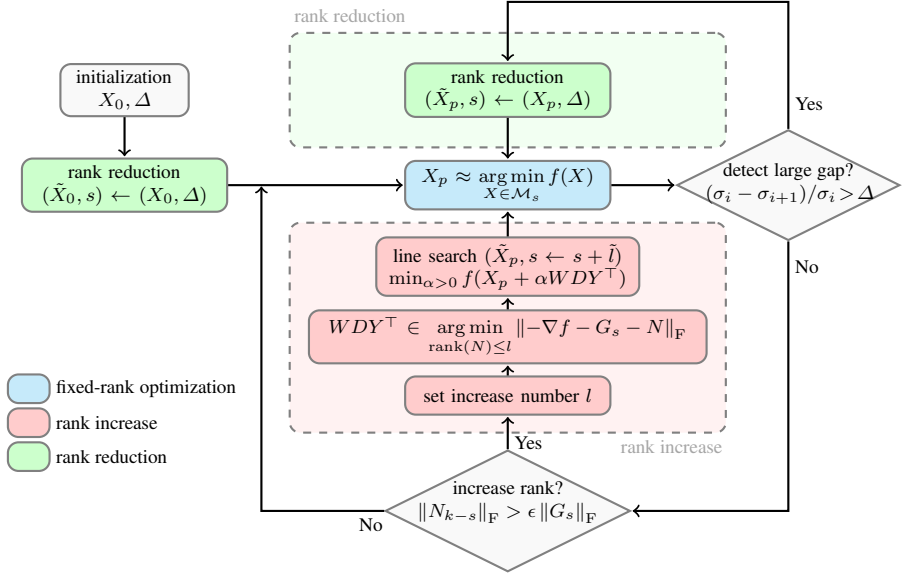

To this end, a globally convergent Riemannian optimization algorithm on $\Ms$ can be employed on the stage of fixed-rank optimization. After this, one can check the violation of optimality~\eqref{eq:critical-two}. If it is still large, which means the current rank $s$ cannot reflect the adequate information for problem~\eqref{prob:bounded-rank}, then we increase the rank by searching along the normal space. In Figure~\ref{fig:flowchart}, we draw the flowchart of the proposed rank-adaptive method (RRAM). There are three major parts in the framework: optimization on the fixed-rank manifold; rank increase; rank reduction. In the rest of this section, we address these three aspects.

\subsection{Riemannian optimization on $\Ms$}\label{subsec:fixed-rank}
Very recently, the Riemannian gradient method with non-monotone line search and Barzilai--Borwein  (BB) step size \cite{BB} has been shown {{to be}} efficient in various applications; see~\cite{iannazzo2018riemannian,hu2019brief,gao2020riemannian}. In addition, its global convergence has been established (e.g., see~\cite[Theorem 5.6]{gao2020riemannian}). We adopt this method to solve the problem \eqref{prob:fixed-s} in Figure~\ref{fig:flowchart}. Given the initial point $\tilde{X}_p\in\Ms$, the detailed method called RBB is listed in Algorithm~\ref{alg:non-monotone gradient}.

\begin{algorithm2e}[htbp]
	\caption{Riemannian gradient method with non-monotone line search and Barzilai--Borwein step size (RBB)}
	\label{alg:non-monotone gradient}
	\textbf{Input:}  $X^{(0)}=\tilde{X}_p\in\Ms$. \\
	\textbf{Require:}  
	$\beta, \delta\in(0,1)$, $\theta \in [0,1]$,
	$0<\gamma_\mathrm{min}<\gamma_\mathrm{max}$, 
	$c_0 = f(X^{(0)})$, $q_0=1$, $\gamma_0>0$.
	\\
	\For{$j=0,1,2,\dots,j_{\max}$}
	{
		1: Choose $Z^{(j)}=-\rgrads{X^{(j)}}$. 
		
		2: Choose a trial step size $\gamma_j= \frac{\jkh{S^{(j-1)},S^{(j-1)}}}{\abs{\jkh{S^{(j-1)},{K^{(j-1)}}}}}$ for odd $j$, otherwise $\frac{\abs{\jkh{S^{(j-1)},K^{(j-1)}}}}{{\jkh{K^{(j-1)},{K^{(j-1)}}}}}$,
		where 
		$S^{(j-1)} = t_{j-1} \mathcal{T}_{X^{(j-1)}\to X^{(j)}}(Z^{(j-1)})$, $K^{(j-1)} =\mathcal{T}_{X^{(j-1)}\to X^{(j)}}(Z^{(j-1)}) - Z^{(j)}$.
		{Set $\gamma_j=\max(\gamma_\mathrm{min},\min(\gamma_j,\gamma_\mathrm{max}))$}. 
		
		3: Find the smallest integer $h$ such that the non-monotone condition
		\begin{equation*}\label{eq:non-monotone}
		f\dkh{\projec_{\Ms}(X^{(j)}+\gamma_j \delta^h Z^{(j)})} \le c_j + \beta \gamma_j \delta^h \jkh{\rgrads{X^{(j)}},Z^{(j)}}
		\end{equation*}
		holds. 
		Set $t_j=\gamma_j \delta^h$.  
		
		4: Set $X^{(j+1)} = \projec_{\Ms}(X^{(j)}+t_j Z^{(j)})$.
		
		5: Update 
		\begin{align*}\label{eq:non-monotone-CQ}
		q_{j+1} &= \theta q_{j} + 1,\\
		c_{j+1} &= \frac{\theta q_{j}}{q_{j+1}} c_{j}  + \frac{1}{q_{j+1}} f(X^{(j+1)}).
		\end{align*}
	}
	\textbf{Output:} Final iterate $X_p=X^{(j_{0})}$. 
\end{algorithm2e}

In step~2, we calculate a step size based on the Riemannian BB method \cite{iannazzo2018riemannian}, and the vector transport on $\Ms$ is defined as 
\begin{equation*}
\mathcal{T}_{X\to Y}: \mathrm{T}_X\Ms\to\mathrm{T}_Y\Ms, Z\mapsto \mathcal{P}_{\mathrm{T}_Y\Ms}(Z).
\end{equation*}
The non-monotone line search is presented in step~3 and step~5. The metric projection $\projec_{\Ms}$ is chosen as the retraction map in step~4, which can be calculated by a truncated SVD. Note that this projection is not necessarily unique, but we can always choose one in practice.  All these calculations in Algorithm~\ref{alg:non-monotone gradient} can be efficiently achieved by exploiting the low-rank structure of $X^{(j)}=U\Sigma V\zz$. The interested readers are referred to~\cite{vandereycken2013low} for detailed implementations. 

We terminate Algorithm~\ref{alg:non-monotone gradient} when the maximum iteration number $j_{\max}$ is reached or other stopping criteria in section~\ref{sec:experiments} are satisfied. Regarding the time efficiency of Algorithm~\ref{alg:non-monotone gradient}, it often compares favorably  with other methods for fixed-rank optimization; see numerical examples in subsection~\ref{subsec:fixed-rank-numerical}. 

\subsection{Rank increase}
In the flowchart of RRAM (Figure~\ref{fig:flowchart}), if $s=k$, we do not increase the rank. Otherwise, given $X_p\in\Ms$, we check the condition
\begin{equation}\label{eq:check-increase}
\norm{N_{k-s}(X_p)}\ff 
> \epsilon\norm{G_s(X_p)}\ff.
\end{equation}
If it holds, then it means that a significant part of $G_{\leq k}(X_p)$ is in the normal space to $\mathcal{M}_s$, in which case we consider that the current rank $s$ is too small.
In order to increase the rank of $X_p$, we propose a ``normal correction" step by searching along the normal space. Specifically, we consider a normal vector 
$$\projec_{\dkh{\TX\Ms}^{\perp}} (-\nabla f(X_p))=-\nabla f(X_p)-G_s (X_p)\in\dkh{\TX\Ms}^{\perp}$$
and choose the best rank-$l$ approximation
\begin{equation}\label{eq:normal-correction}
WDY\zz \in \argmin_{\rank(N)\le l} \norm{-\nabla f(X_p)-G_s (X_p)-N}\ff
\end{equation}
such that $\tilde{l}\le l$, $W\in\stiefel(\tilde{l},m)$, $Y\in\stiefel(\tilde{l},n)$, and $D\in\R^{\tilde{l}\times \tilde{l}}$ is a diagonal matrix that has full rank.
Note that it is equivalent to a $l$-truncated SVD of $-\nabla f(X_p)-G_s (X_p)$.

To verify the validity of ``normal correction" step, we will need the following result.





\begin{proposition}\label{prop:WY}
	Given $X=U\Sigma V\zz\in\Ms$, every best rank-$l$ approximation $WDY\zz$ of $-\nabla f(X)-G_s (X)$ satisfies that
	$$W\zz U=0\quad \mbox{and} \quad Y\zz V=0.$$  Moreover, if $-\nabla f(X)-G_s (X)\neq0$, then $WDY\zz$ is a descent direction for $f$, i.e., 
	\begin{align*}
		\left.\frac{\mathrm{d}}{\mathrm{d}t} f(X+tWDY\zz)\right|_{t=0} =-\norm{D}\fs<0.
	\end{align*}
\end{proposition}
\begin{proof}
	It follows from~\eqref{eq:normal-proj} and \eqref{eq:Gs_define} that  
	\begin{align*}
	H(X)&:=-\nabla f(X)-G_s (X) =-\nabla f(X)-\projec_{\TX\Ms}(-\nabla f(X))  \\
	&{~}=\projec_{\dkh{\TX\Ms}^{\perp}} (-\nabla f(X)) = -P^{\perp}_U \nabla f(X) P^{\perp}_V\\
	&{~}= -U_\perp U_\perp\zz \nabla f(X) V_\perp V_\perp\zz.
	\end{align*}
	Therefore, any SVD of $H(X)$ has the form $U_\perp \bar{U} \bar{\Sigma} \bar{V}\zz V_\perp\zz$, where
	\begin{align*}
		\bar{U}&=[\bar{u}_1,\dots,\bar{u}_r]\in\stiefel(r,m-s), \\
		\bar{V}&=[\bar{v}_1,\dots,\bar{v}_r]\in\stiefel(r,n-s), \\
		\bar{\Sigma}&=\diag(\bar{\sigma}_1,\dots,\bar{\sigma}_r)
	\end{align*} 
	with $\bar{\sigma}_1\ge\cdots\ge\bar{\sigma}_r>0$ and $r=\rank(H(X))$. 
	It follows that 
	\begin{align*}\label{eq:HX}
	H(X)= -U_\perp U_\perp\zz \nabla f(X) V_\perp V_\perp\zz 
	=U_\perp \bar{U} \bar{\Sigma} \bar{V}\zz V_\perp\zz
	= \sum_{i=1}^{r} \bar{\sigma}_i \dkh{U_{\perp}\bar{u}_i} \dkh{V_{\perp}\bar{v}_i}\zz,
	\end{align*}
	which is any compact SVD of $H(X)$. Note that $WDY\zz$ is a best rank-$l$ approximation of $H(X)$ iff $WDY\zz$ is a $l$-truncated SVD of $\sum_{i=1}^{r} \bar{\sigma}_i \dkh{U_{\perp}\bar{u}_i} \dkh{V_{\perp}\bar{v}_i}\zz$, i.e., $\sum_{i=1}^{\tilde{l}} \bar{\sigma}_i \dkh{U_{\perp}\bar{u}_i} \dkh{V_{\perp}\bar{v}_i}\zz$ where $\tilde{l}:=\min(l,r)$. Let 
	\begin{align*}
	W&=U_{\perp}[\bar{u}_1,\dots,\bar{u}_{\tilde{l}}], \\
	Y&=V_{\perp}[\bar{v}_1,\dots,\bar{v}_{\tilde{l}}], \\
	D&=\diag(\bar{\sigma}_1,\dots,\bar{\sigma}_{\tilde{l}}).
	\end{align*} 
	It yields that $W\zz U=0$ and $Y\zz V=0$. 
	
	In addition, it follows from $U_\perp\zz \nabla f(X) V_\perp=-\bar{U} \bar{\Sigma} \bar{V}\zz$ that 
	\begin{align*}
	\left.\frac{\mathrm{d}}{\mathrm{d}t} f(X+tWDY\zz)\right|_{t=0} &= \jkh{\nabla f(X), WDY\zz} \\
	&= \jkh{\nabla f(X), U_{\perp} \dkh{\sum_{i=1}^{\tilde{l}} \bar{\sigma}_i \bar{u}_i \bar{v}_i\zz} V_{\perp}\zz }\\
	&= \jkh{U_{\perp}\zz \nabla f(X)V_{\perp},  {\sum_{i=1}^{\tilde{l}} \bar{\sigma}_i \bar{u}_i \bar{v}_i\zz}} \\
	&= - \jkh{\bar{U} \bar{\Sigma} \bar{V}\zz, \sum_{i=1}^{\tilde{l}} \bar{\sigma}_i \bar{u}_i \bar{v}_i\zz}\\
	&=-\sum_{i=1}^{\tilde{l}}\bar{\sigma}_i^2=-\norm{D}\fs<0.
	\end{align*}
	Thus, $WDY\zz$ is a descent direction.\qed
\end{proof}

In view of Proposition~\ref{prop:WY} and~\eqref{eq:normal-space}, $WDY\zz$ is a descent direction in $\dkh{\TX\Ms}^{\perp}$. Hence, we can perform a line-search
\begin{equation*}
\min\limits_{\alpha> 0}~f(X_p+\alpha WDY\zz).
\end{equation*}
As the objective function $f(X)$ in~\eqref{prob:fixed-rank} is quadratic, this problem has the closed-form solution
\begin{equation}\label{eq:stepsize-alpha}
\alpha^* = -\frac{\jkh{\po(WDY\zz),\po(X-A)}}{\norm{\po(WDY\zz)}\fs}.
\end{equation}
Moreover, given the low-rank structure of $X_p=U\Sigma V\zz$, it readily follows from Proposition~\ref{prop:WY} that 
\begin{align*}
X_+:=X_p+\alpha WDY\zz &= U\Sigma V\zz + \alpha WDY\zz
= \begin{bmatrix}
U & W
\end{bmatrix} \begin{bmatrix}
\Sigma   & 0\\
0& \alpha D  
\end{bmatrix} \begin{bmatrix}
V &Y
\end{bmatrix}\zz
\end{align*}
has rank $s+\tilde{l}$. 
In addition, the SVD of $X_+$, denoted by $X_+=U_+ \Sigma_+ V_+\zz$, can be directly assembled by sorting the columns of $[U~W]$, $\fkh{\begin{smallmatrix} 
	\Sigma   & 0\\
	0& \alpha D  
	\end{smallmatrix}}$ and $[V~Y]$ according to the descent order of diagonal entries in  $\fkh{\begin{smallmatrix} 
	\Sigma   & 0\\
	0& \alpha D  
	\end{smallmatrix}}$.

We summarize all these steps in Algorithm~\ref{alg:rank-increase}.
In practice,  the rank increase number $l$ is set to a constant.

\begin{algorithm2e}[!h]
	\caption{Rank increase}
	\label{alg:rank-increase}
	\textbf{Input:}  $X_p=U\Sigma V\zz \in\Ms$, a rank increase number $l\le k-s$.\\
	{
		1: Compute $W$, $D$ and $Y$ by~\eqref{eq:normal-correction}.\\
		2: Compute the step size $\alpha$ by~\eqref{eq:stepsize-alpha}.\\
		3: Sort the columns of $[U~W], \fkh{\begin{smallmatrix}
			\Sigma & 0\\0 &  \alpha D
			\end{smallmatrix}}, [V~Y]$ by descending the diagonal entries of  $\fkh{\begin{smallmatrix}
			\Sigma & 0\\0 &  \alpha D
			\end{smallmatrix}}$, and accordingly obtain  $U_+$, $\Sigma_+$, $V_+$.
	}
	
	\textbf{Output:} $\tilde{X}_p := U_+ \Sigma_+ V_+\zz$ with the rank $s+\tilde{l}$, where $\tilde{l} := \rank(\tilde{X}_p) - s$. 
\end{algorithm2e}

A quick verification of rank increase is presented in Figure~\ref{fig:motivation_increase}. We consider a method that combines the fixed-rank optimization (RBB) with the rank-one increase. Figure~\ref{fig:motivation_increase} reports the evolution of $\norm{N_{k-s}(X)}\ff$ and $\norm{G_s(X)}\ff$ for solving problem~\eqref{prob:bounded-rank} with a rank-one initial point. It is observed that when the stage of fixed-rank optimization is finished,  $\norm{G_s(X)}\ff$ dramatically degrades while $\norm{N_{k-s}(X)}\ff$ does not. Moreover, increasing the rank will lead to a reduction on  $\norm{N_{k-s}(X)}\ff$, which verifies the effect of rank increase.

\begin{figure}[htbp]
	\centering
	{\includegraphics[width=.57\textwidth]
		{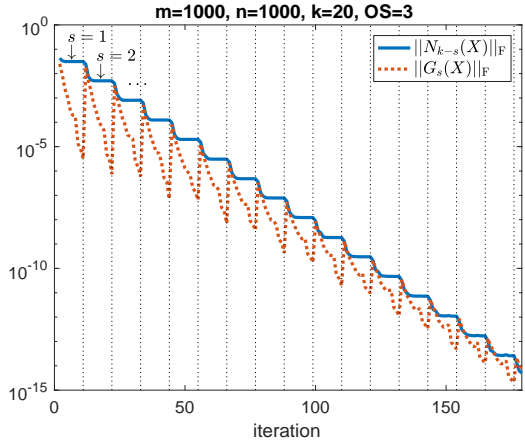}}
	\caption{Effect of the rank increase}
	\label{fig:motivation_increase}
\end{figure}

\subsection{Rank reduction}\label{subsec:rank-reduction}
As $\Ms$ is not closed, an iterate sequence in $\Ms$ may have limit points of rank less than $s$. If the iterates are found to approach $\mathcal{M}_{\leq (s-1)}$, then it makes sense to solve the optimization problem on the set of smaller fixed-rank matrices. In addition, it can reduce the dimension of problem and thereby save the memory. 

One possible strategy to decrease the rank has been proposed in~\cite{Zhou2016riemannian}. Specifically, given $X_p=U\Sigma V\zz$ with $\Sigma=\diag(\sigma_1,\dots,\sigma_s)$ and $\sigma_1\ge\cdots\ge\sigma_s>0$, given a threshold $\Delta>0$, one can set to zero the singular values smaller than $\sigma_1 \Delta$. This rank reduction step returns a matrix $\tilde{X}_p\in\manifold_{\tilde{r}}$ by the best rank-$\tilde{r}$ approximation of $X_p$, where $\tilde{r}:=\max\hkh{i:\sigma_i\ge\sigma_1 \Delta}$.

In practice, we observe that the gap of singular values also plays an important role in fixed-rank optimization; see the numerical examples in subsection~\ref{subsec:rank-reduction-numerical}. Therefore, we propose to check if there is a large gap in the singular values and, if so, to decrease the rank accordingly.
To this end, we consider a criterion based on the relative change of singular values
\begin{equation}\label{eq:gap_sv-0}
{\frac{\sigma_i-\sigma_{i+1}}{\sigma_i}}>\Delta,
\end{equation}
where $\Delta\in(0,1)$ is a given threshold. Figure~\ref{fig:svd} presents two typical distributions of singular values, and ${\sigma_i}/{\sigma_1}$, ${(\sigma_i-\sigma_{i+1})}/{\sigma_i}$ are also computed. We observe that the large gap of singular values in the first row can be detected by~\eqref{eq:gap_sv-0} with setting $\Delta=0.1$, while the condition $\sigma_i<\sigma_1 \Delta$ is not activated. 

\begin{figure}[htpb]
	\centering
	{\includegraphics[width=.325\textwidth]
		{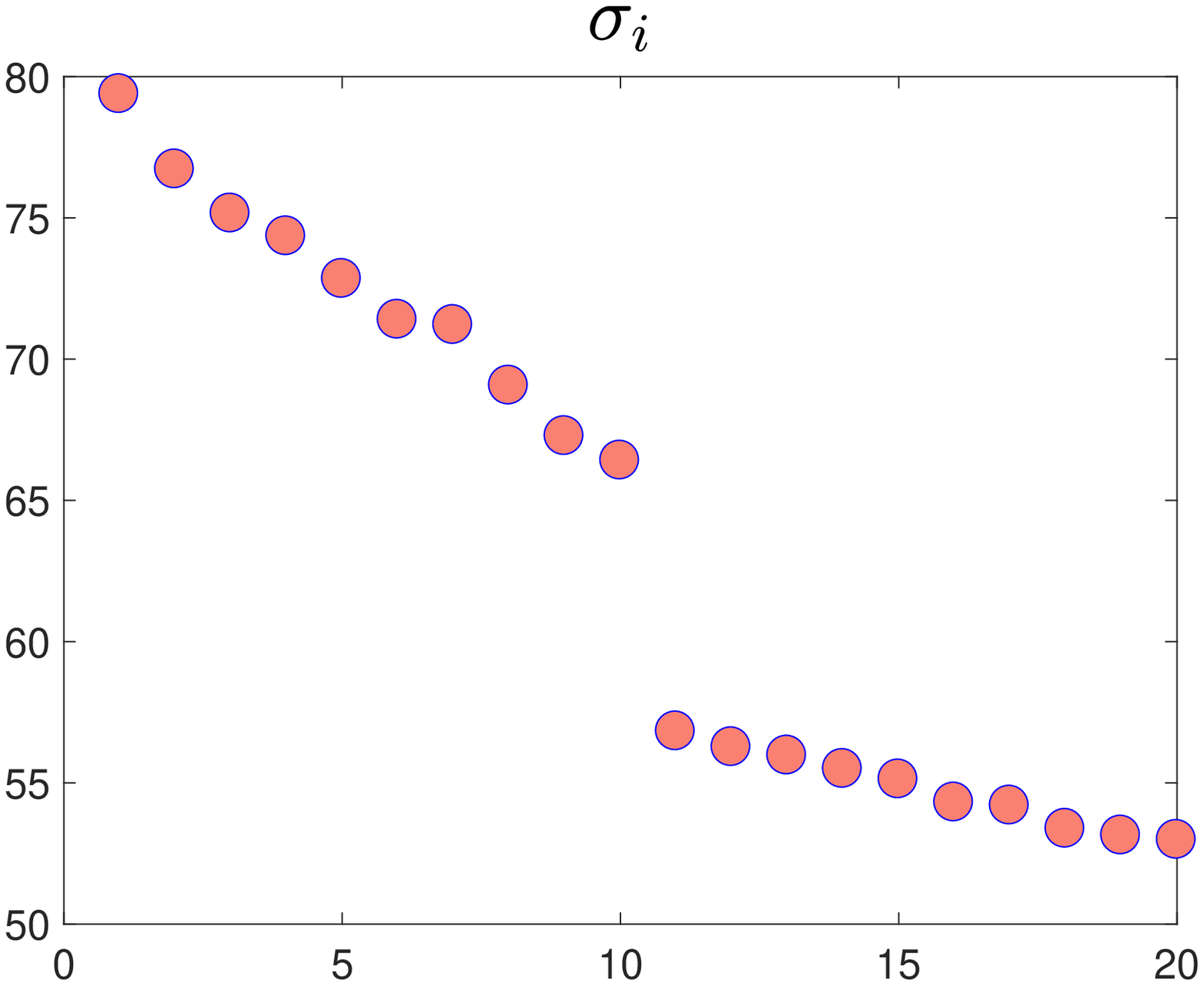}}
	{\includegraphics[width=.325\textwidth]
		{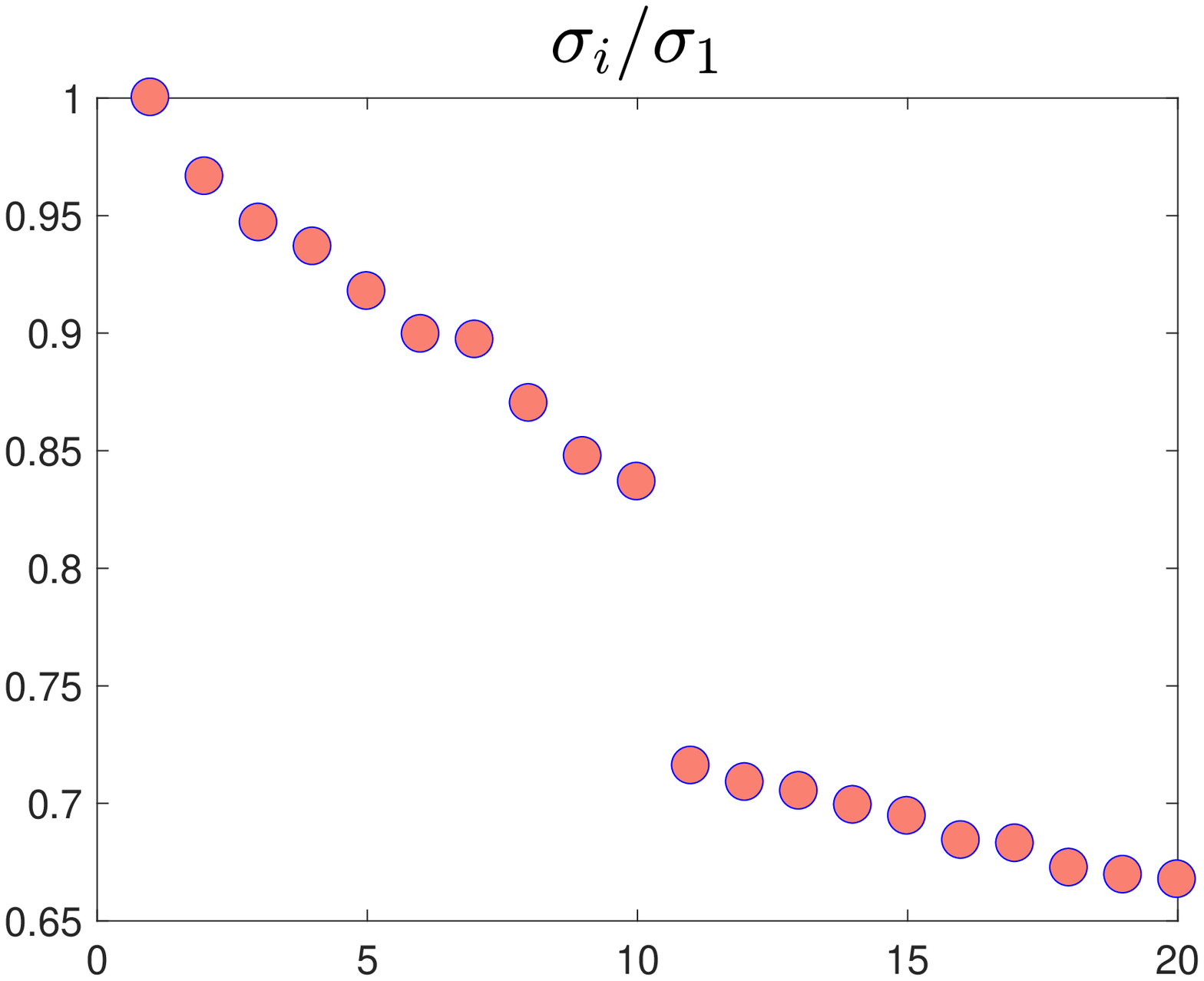}}
	{\includegraphics[width=.325\textwidth]
		{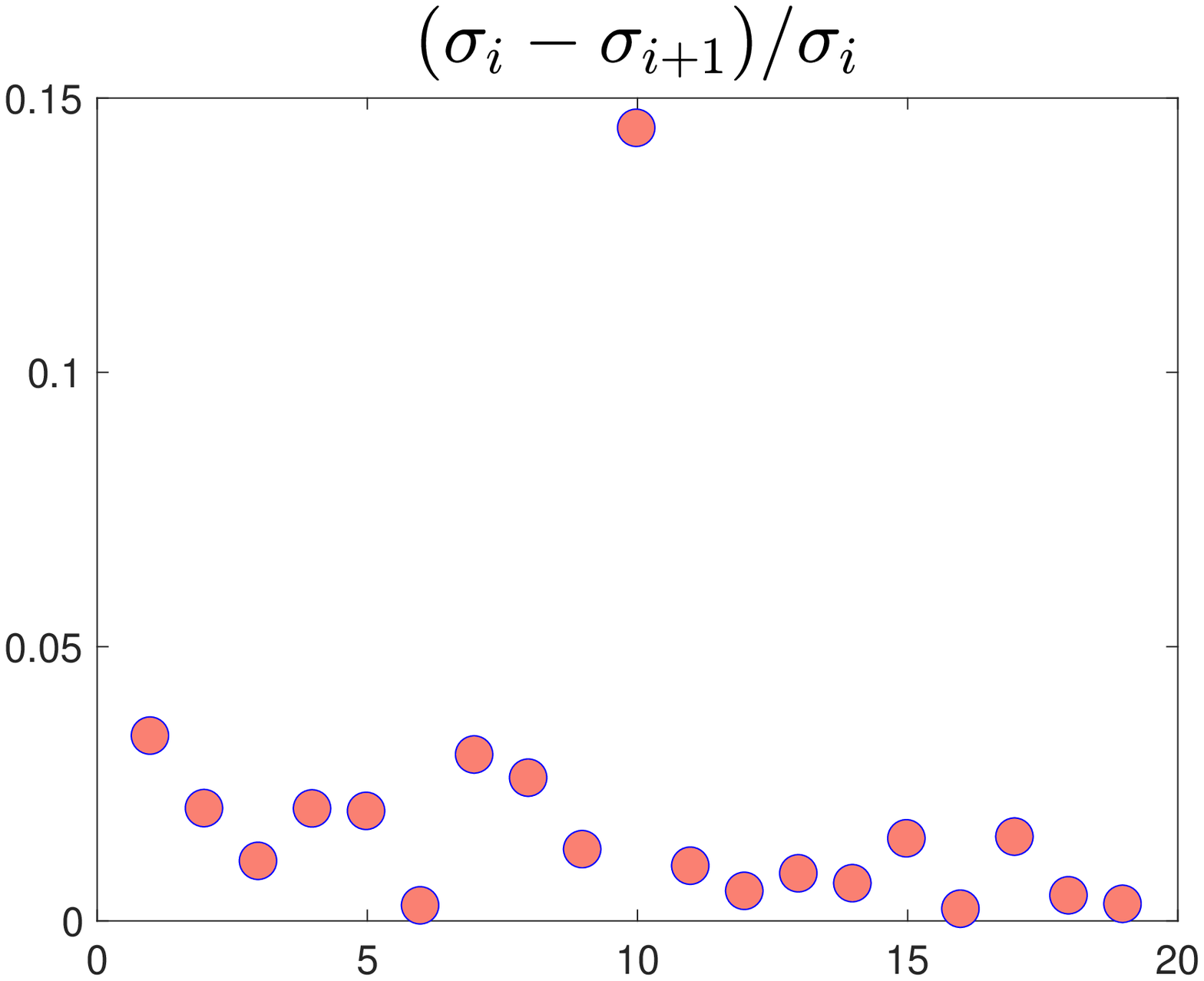}}
	
	{\includegraphics[width=.325\textwidth]
		{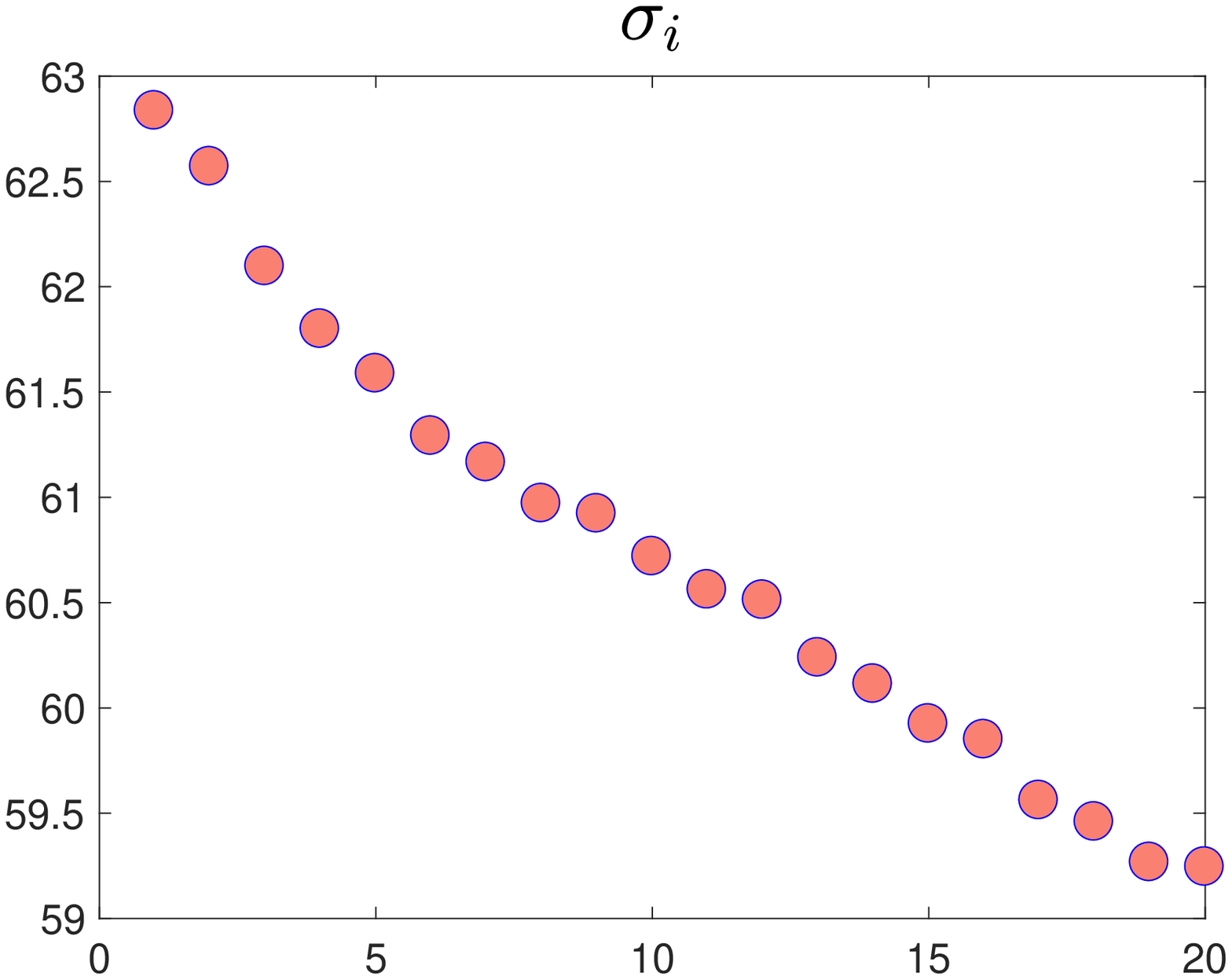}}
	{\includegraphics[width=.325\textwidth]
		{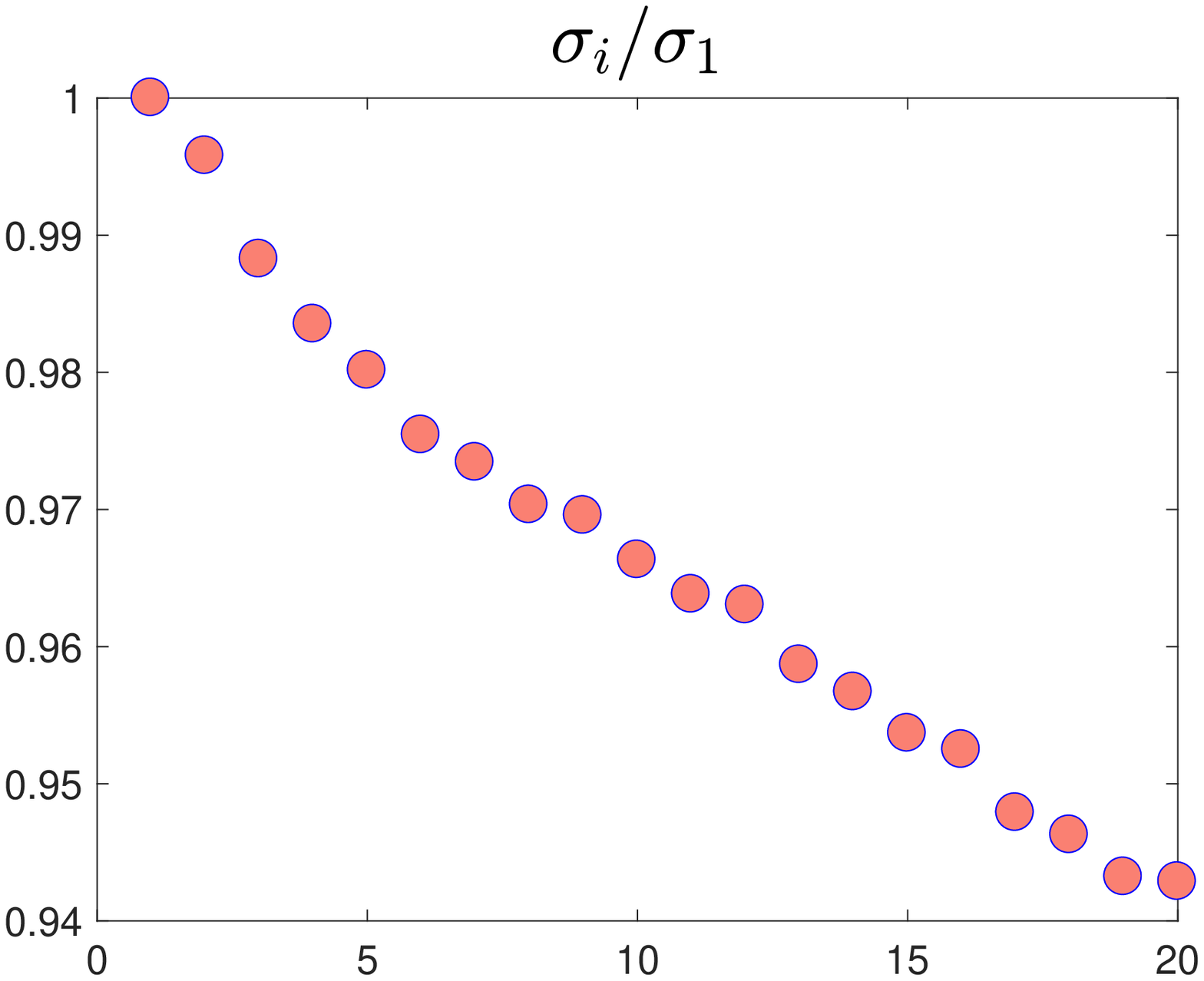}}
	{\includegraphics[width=.325\textwidth]
		{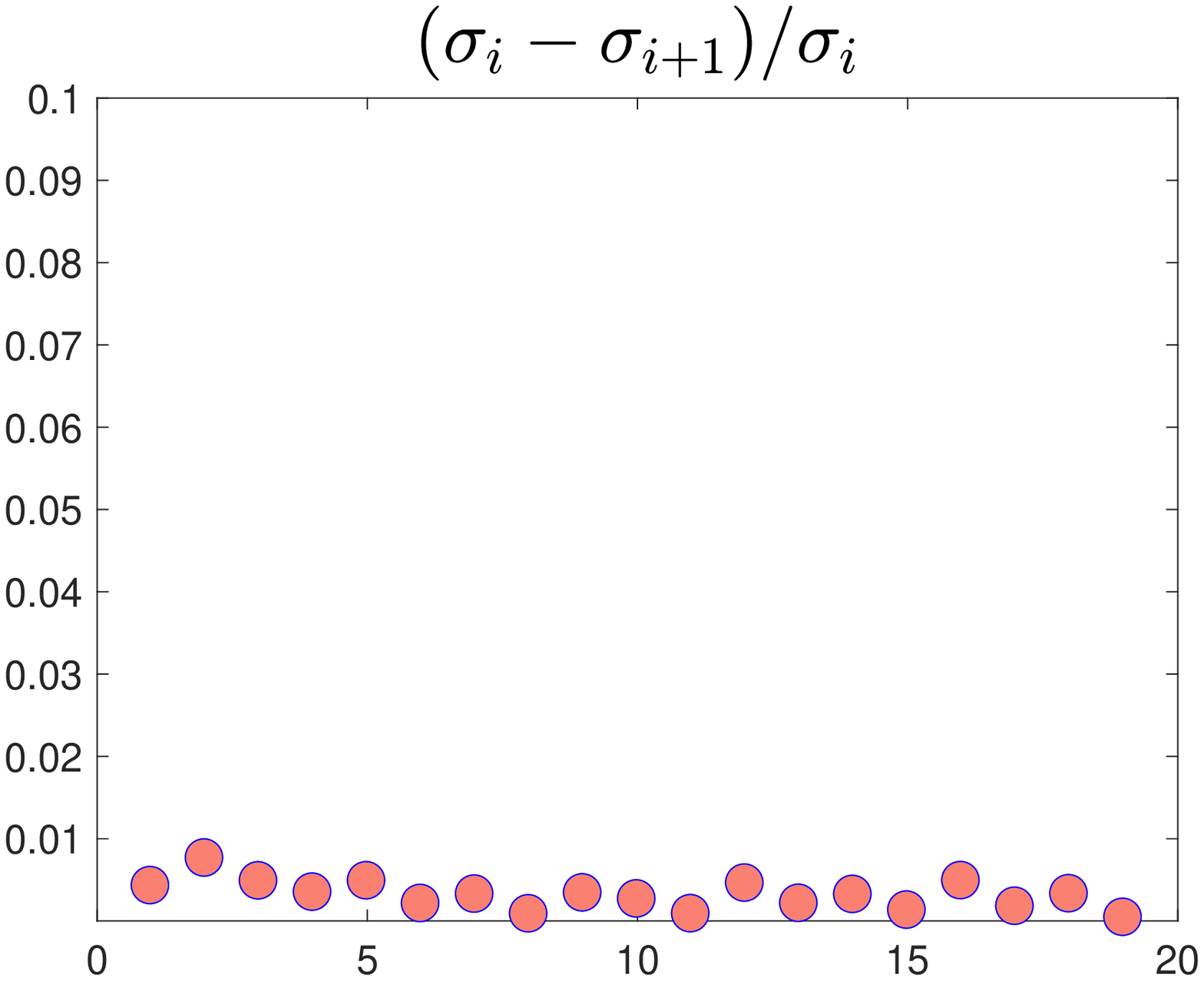}}
	\caption{Singular values of $X=\projec_{\manifold_{20}}(\po(A))$. First row: $A=LR\zz\in\R^{1000\times 1000}$, where $L, R\in\R^{1000\times 10}$. Second row: $A=\texttt{randn}(1000,1000)\in\R^{1000\times 1000}$.  \label{fig:svd}}
\end{figure}

In order to avoid losing too much information, we do not reduce the rank aggressively. The large gap of singular values is only detected by finding the index $\tilde{r}$ such that 
	\begin{equation}\label{eq:gap_sv}
		\tilde{r} := \left\{
		\begin{array}{ll}
			s, & \quad\mbox{if~} \max_{i}\hkh{\frac{\sigma_i-\sigma_{i+1}}{\sigma_i}}\leq\Delta;\\
			\argmax_{i}\hkh{\frac{\sigma_i-\sigma_{i+1}}{\sigma_i}}, & \quad\mbox{otherwise}.
		\end{array}
		\right.
	\end{equation}
Note that if there is no large gap detected in terms of the threshold $\Delta$, then we do not reduce the rank, i.e., $\tilde{r}=s$. Otherwise, we choose the index that maximizes the gap ${(\sigma_i-\sigma_{i+1})}/{\sigma_i}$. A benefit of taking maximum of ${(\sigma_i-\sigma_{i+1})}/{\sigma_i}$ is that we can exclude some undesirable cases that the top one or few singular values are large and quite separated from a bulk of singular values. For instance, if singular values are distributed as $\{1,0.89,0.88,\dots,0.8,0.1,0.09,\dots,0.01\}$, the condition~\eqref{eq:gap_sv} returns a reasonable gap detection between $0.8$ and $0.1$ with $\Delta=0.1$. However, if we consider the smallest $i$ such that ${(\sigma_i-\sigma_{i+1})}/{\sigma_i}>0.1$, we will drop all singular values except for the largest one, which is too aggressive. 
	
	The proposed rank reduction step is listed in Algorithm~\ref{alg:rank-reduction}. 
	
	\begin{algorithm2e}[htbp]
		\caption{Rank reduction}
		\label{alg:rank-reduction}
		\textbf{Input:}  $X_p=U\Sigma V\zz \in\Ms$, $\Delta\in(0,1)$.\\
		{
			1: Choose the index $\tilde{r}$ by~\eqref{eq:gap_sv}.\\
			2. Choose $U_{\tilde{r}}, V_{\tilde{r}}$ as the first $\tilde{r}$ columns of $U$, $V$, and set $\Sigma_{\tilde{r}}=\diag(\sigma_1,\dots,\sigma_{\tilde{r}})$.
		}
		
		\textbf{Output:} $\tilde{X}_p := U_{\tilde{r}}\Sigma_{\tilde{r}} V_{\tilde{r}}\zz$ with the rank $\tilde{r}$.
	\end{algorithm2e}

Algorithm~\ref{alg:rank-reduction} is just one among many possible rank reduction strategies that retain the ``largest'' singular values and set the other ones to zero. The performance of those strategies is highly problem-dependent. Nevertheless, without such a strategy, the rank-adaptive method may miss opportunities to reduce the rank and thereby benefit from a reduced computational cost. Moreover, in order to address the issue, mentioned in subsection~\ref{subsec:related-work}, that $\mathcal{M}_k$ is not closed, some rank reduction mechanism is required to rule out sequences with infinitely many points in $\mathcal{M}_{>s}$ and a limit point in $\mathcal{M}_s$.


\begin{remark}\label{remark:tab}
	In comparison with the rank-related algorithms in Table~\ref{tab:rank-related},  the proposed rank-adaptive method in Figure~\ref{fig:flowchart} has several novel aspects.
	Firstly, as an inner iteration for the fixed-rank optimization, the RBB method with non-monotone line search tends to outperform other Riemannian methods such as RCG; see subsection~\ref{subsec:fixed-rank-numerical} and~\ref{subsec:ablation-numerical}. Secondly, we search along the normal space to increase the rank. 
	This contrasts with~\cite{Zhou2016riemannian}, where the update direction is obtained by projecting the antigradient onto a tangent cone. Moreover, in contrast with~\cite[\S III.A]{Uschmajew_V:2015}, we do not assume the fixed-rank algorithm (Algorithm~\ref{alg:non-monotone gradient}) to return a point $X_p$ that satisfies $G_s(X_p)=0$; however, if it does, then the proposed rank increase step coincides with the update $X_+ = X +\alpha G_{\le (s+l)}(X)$ of~\cite{Uschmajew_V:2015} in view of~\eqref{eq:critical-two}.
	Thirdly, the proposed rank reduction mechanism differs from the one in~\cite{Zhou2016riemannian}, as explained in subsection~\ref{subsec:rank-reduction}. Its performance is illustrated in subsection~\ref{subsec:rank-reduction-numerical}.
\end{remark}

\section{Numerical experiments}\label{sec:experiments}
In this section, we first demonstrate the effectiveness of the proposed rank-adaptive algorithm, and then compare it with other methods on low-rank matrix completion problems. For simplicity, we restrict our comparisons to manifold-based methods since they usually perform well on this problem and are comparable with other low-rank optimization methods; see~\cite{vandereycken2013low}.

Unless otherwise mentioned, the low-rank matrix $$A=LR\zz\in\R^{m\times n}$$ in~\eqref{prob:bounded-rank} is generated by two random matrices $L\in\R^{m\times r}$, $R\in\R^{n\times r}$ with i.i.d. standard normal distribution. The sample set $\Omega$ is randomly generated on $\{1,\dots,m\}\times\{1,\dots,n\}$ with the uniform distribution of $\abs{\Omega}/(mn)$. The oversampling rate (OS, see~\cite{vandereycken2013low}) for $A$ is defined as the ratio of $\abs{\Omega}$ to the dimension of $\manifold_r$, i.e.,
$$\mathrm{OS}:=\frac{\abs{\Omega}}{(m+n-r)r}.$$
Note that $\mathrm{OS}$ represents the difficulty of recovering matrix $A$, and it should be larger than $1$.

The stopping criteria for different algorithms are based on  the relative gradient 
and relative residual of their solutions, also the relative change of function value. Specifically, 
\begin{align*}
\mathrm{relative~gradient}:\quad & \frac{\norm{\rgrads{X}}\ff}{\max(1,\norm{X}\ff)}<\epsilon_g,\\
\mathrm{relative~residual}:\quad &\frac{\norm{\po(X)-\po(A)}\ff}{\norm{\po(A)}\ff}<\epsilon_\Omega,\\
\mathrm{relative~change}:\quad &\abs{1-\frac{\norm{\po(X_i)-\po(A)}\ff}{\norm{\po(X_{i-1})-\po(A)}\ff}}<\epsilon_f.
\end{align*}
Once one of the above criteria or the maximum iteration number $1000$ is reached, we terminate the algorithms. Note that these criteria are introduced in~\cite{vandereycken2013low}. The default tolerance parameters are chosen as $\epsilon_g=10^{-12}$, 
$\epsilon_\Omega = 10^{-12}$, $\epsilon_f = 10^{-4}$. 
The rank increase parameter $\epsilon$ in~\eqref{eq:check-increase} is set to $10$, and the rank increase number $l$ is $1$.
The rank reduction threshold $\Delta$ in~\eqref{eq:gap_sv} is set to $0.1$. The inner maximum iteration number $j_{\max}$ is set to 100, and other parameters in Algorithm~\ref{alg:non-monotone gradient} are the same as those in~\cite{gao2020riemannian}.

All experiments are performed on a laptop with 2.7 GHz Dual-Core Intel i5 processor and 8GB of RAM running MATLAB R2016b under macOS 10.15.2. The code that produces the result is available from \url{https://github.com/opt-gaobin/RRAM}.


\subsection{Comparison on the fixed-rank optimization}\label{subsec:fixed-rank-numerical}
Before we test the rank-adaptive method, we first illustrate the performance of {the} RBB method proposed in Algorithm~\ref{alg:non-monotone gradient} on the fixed-rank optimization problem~\eqref{prob:fixed-rank}. We compare RBB with a state-of-the-art fixed-rank method called \texttt{LRGeomCG}\footnote{\label{footnote:RCG}Available from \url{https://www.unige.ch/math/vandereycken/matrix_completion.html}.} \cite{vandereycken2013low}, which is a Riemannian CG method for fixed-rank optimization.

The problem is generated with $m=n=10000, r=40$ and $\mathrm{OS}=3$. The rank parameter $k$ in~\eqref{prob:fixed-rank} is set to $k=r$, which means the true rank of $A$ is provided. The initial point is generated by
\begin{equation}\label{eq:initialization}
	X_0:=\projec_{\mathcal{M}_{k}}(\po(A)).
\end{equation}
Figure~\ref{fig:CGvsBB} reports the numerical results. It is observed that the RBB method performs better than LRGeomCG in terms of time efficiency to achieve comparable accuracy. In addition, one can find the non-monotone property of RBB that stems from the non-monotone line search procedure in Algorithm~\ref{alg:non-monotone gradient}.

\begin{figure}[htbp]
	\centering
	\subfigure[Relative gradient]
	{\includegraphics[width=.48\textwidth]
		{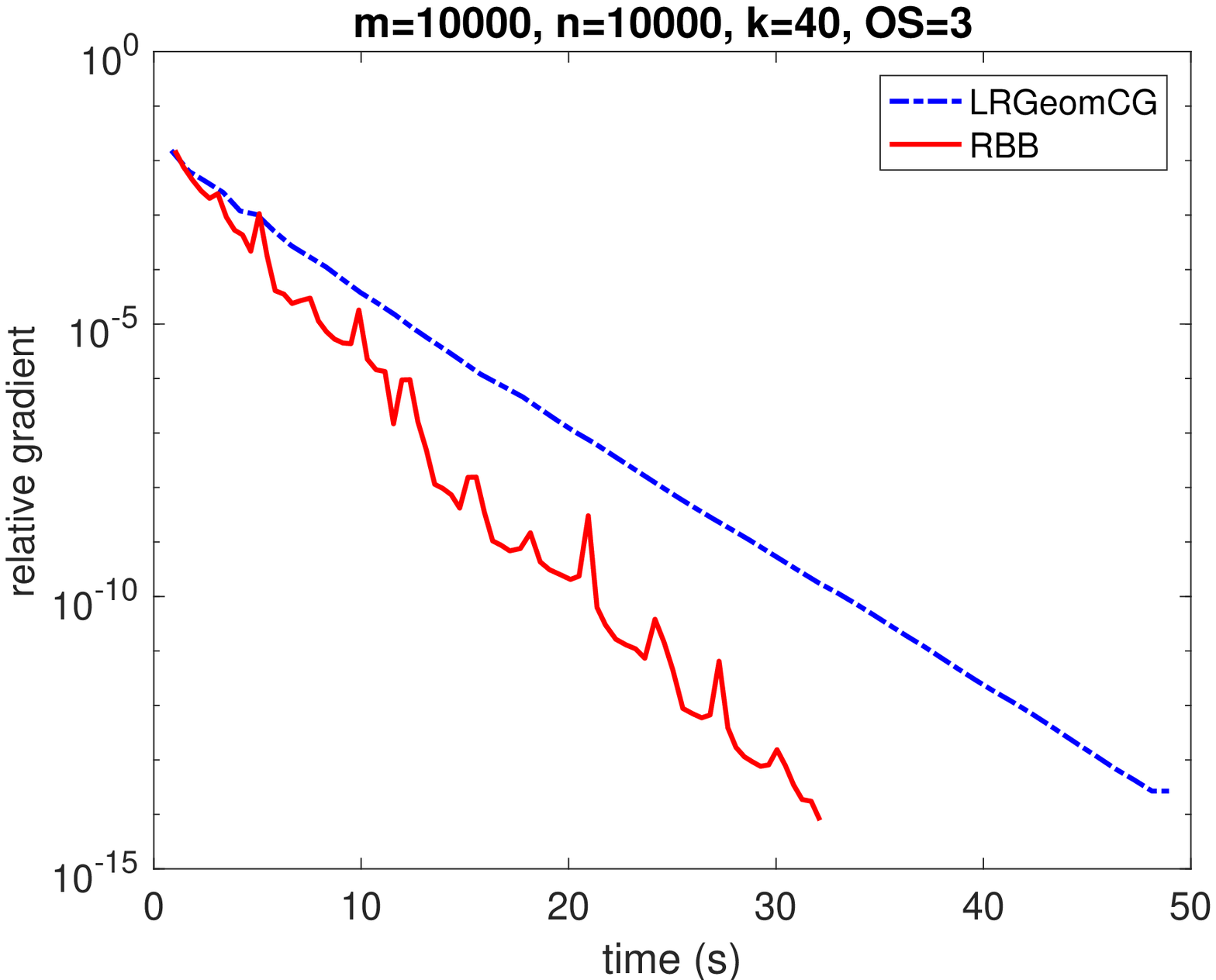}}
	\quad
	\subfigure[Relative residual]
	{\includegraphics[width=.48\textwidth]
		{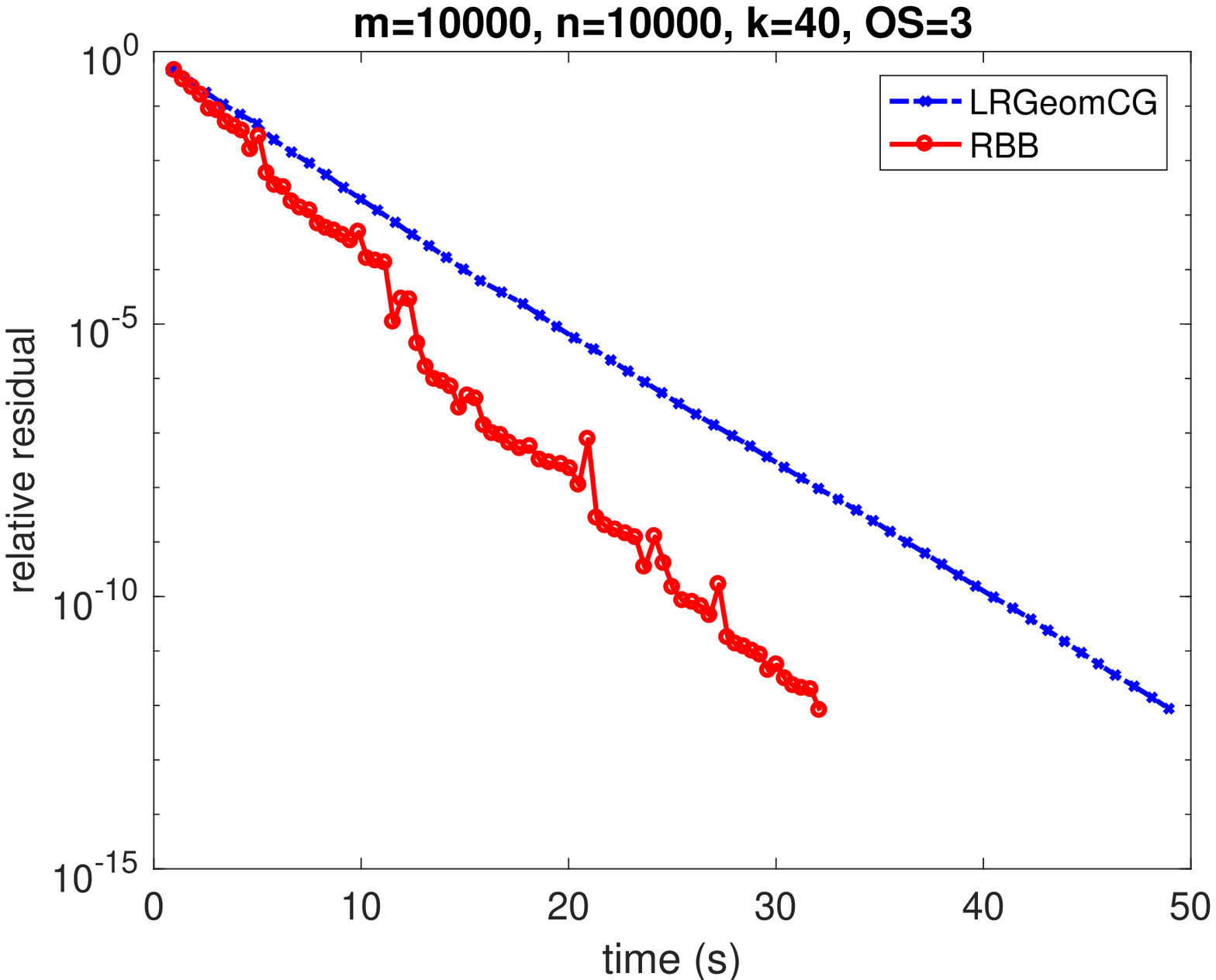}}
	\caption{A comparison on the fixed-rank optimization}
	\label{fig:CGvsBB}
\end{figure}

In order to investigate the performance of RBB on different problems, we test on three datasets with varying $m$, the rank parameter $k$, and $\mathrm{OS}$, respectively. Specifically, we fix the oversampling rate $\mathrm{OS}=3$, $k=20$, and chose $m=n$ from the set $\{2000j: j=1,\dots,10\}$. Alternatively, we choose $k$ from $\{10j: j=1,\dots,8\}$ and fix $m=n=10000$. In addition, the last dataset is varying $\mathrm{OS}$ from $\{1,\dots,10\}$, and choosing $m=n=10000$, $k=20$.  The running time of RBB and LRGeomCG are reported in Figure~\ref{fig:CGvsBB_nk}. Notice that RBB has less running time than LRGeomCG when the size of problem ($m$ and $k$) increases.
Additionally, we observe that RBB 
outperforms LRGeomCG among all different oversampling settings.

\begin{figure}[htbp]
	\centering
	\subfigure[Test on varying $m$]
	{\includegraphics[width=.325\textwidth]
		{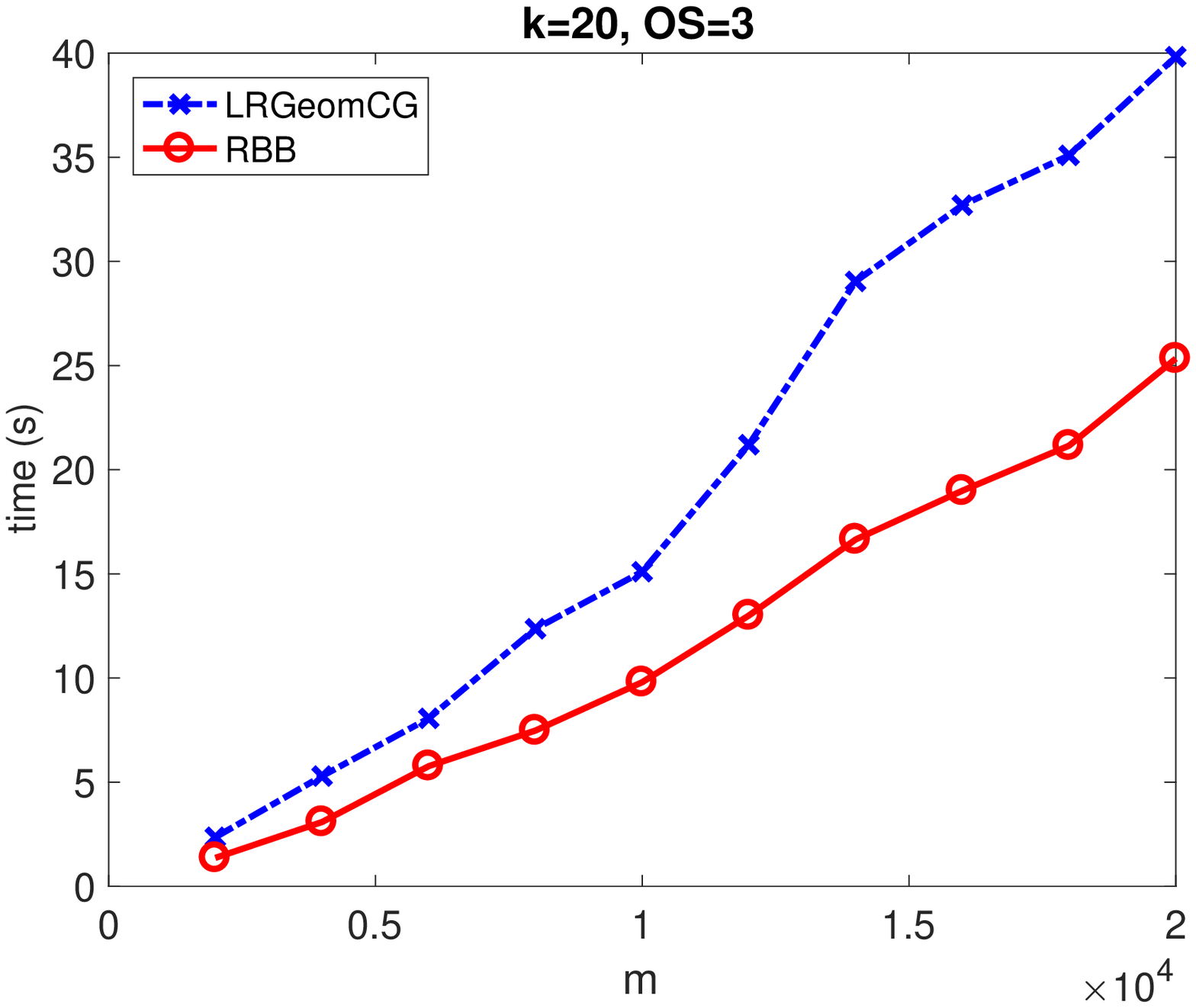}}
	\subfigure[Test on varying $k$]
	{\includegraphics[width=.33\textwidth]
		{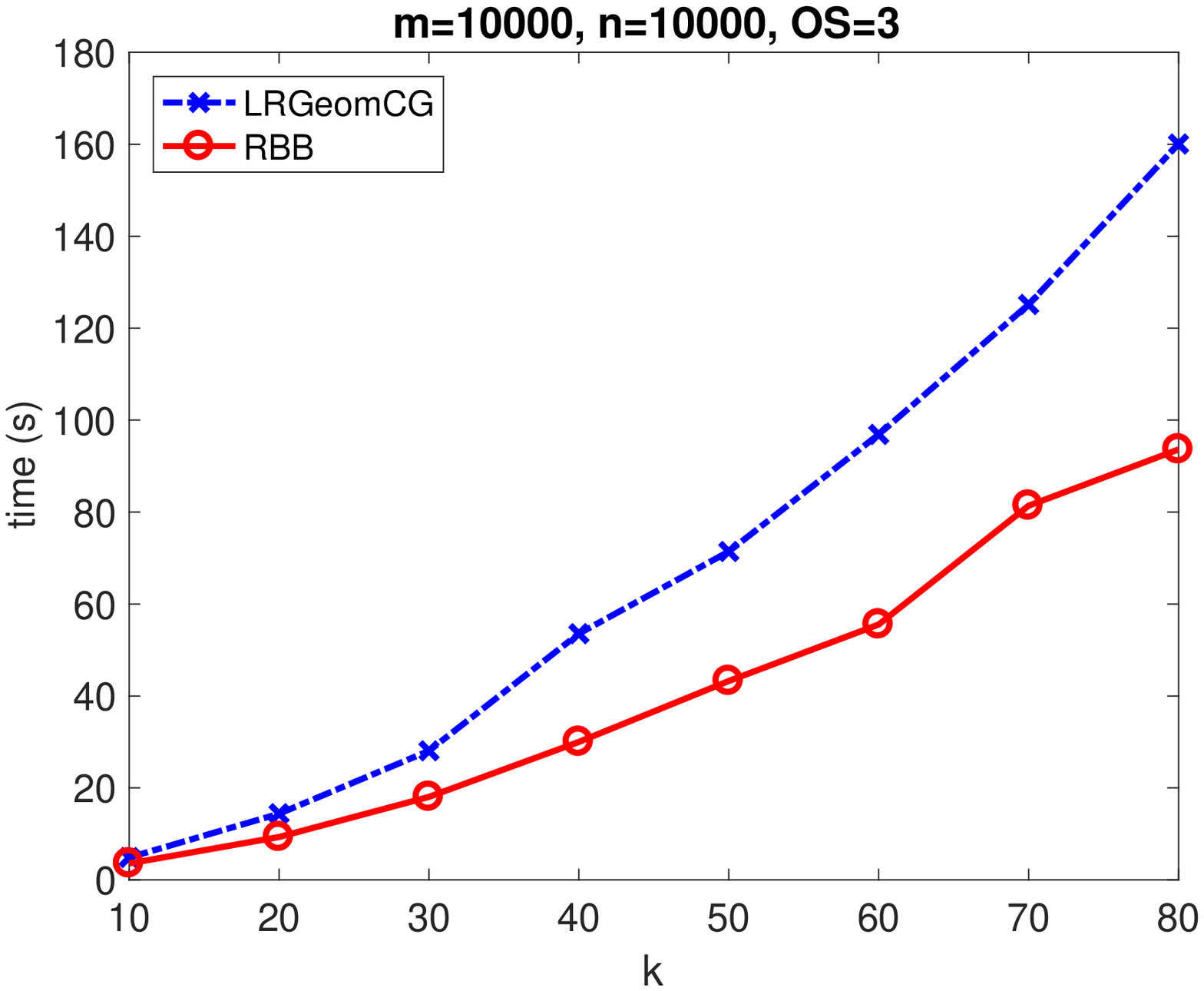}}
	\subfigure[Test on varying $\mathrm{OS}$]
	{\includegraphics[width=.325\textwidth]
		{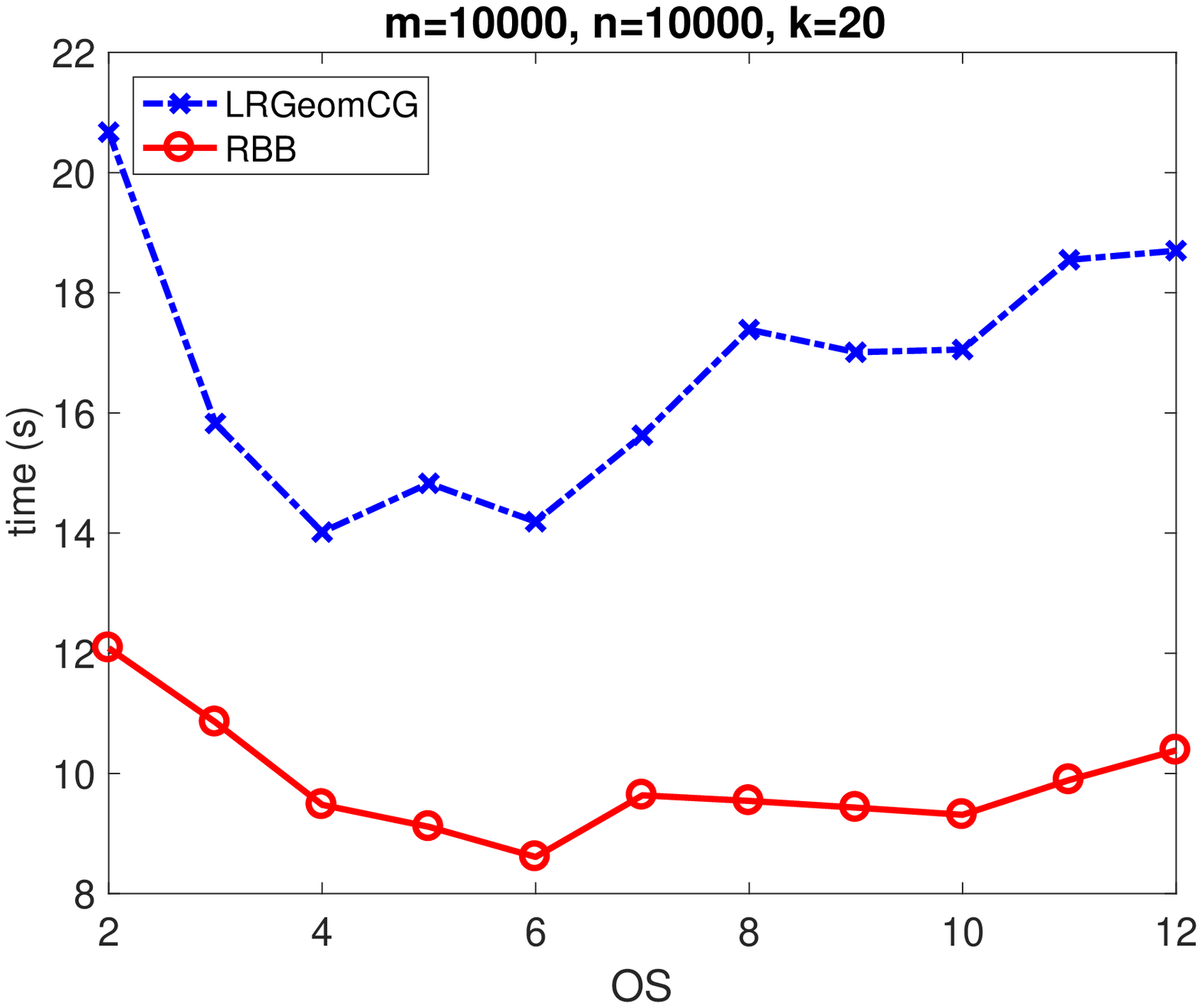}}
	\caption{A comparison on the fixed-rank optimization with varying $m$, $k$ and $\mathrm{OS}$}
	\label{fig:CGvsBB_nk}
\end{figure}

\subsection{Comparison on the rank reduction}\label{subsec:rank-reduction-numerical}
The effectiveness of the rank reduction step in RRAM is verified in this subsection. RRAM combines with the RBB method as the fixed-rank optimization, and we call it \texttt{RRAM-RBB}. For comparison, we also test LRGeomCG to illustrate that the rank-adaptive method is more suitable than fixed-rank methods for low-rank matrix completion. We generate problem~\eqref{prob:bounded-rank} with $m=n=1000$ and $\mathrm{OS}=3$. The data matrix $A=LR\zz$ is randomly generated by rank 10 matrices. The following comparison is twofold based on different initial guesses that have similar singular value distributions in Figure~\ref{fig:svd}. 

In a first set of experiments, the methods are initialized by~\eqref{eq:initialization}, i.e., the best rank-$k$ approximation of $\po(A)$. Given the rank parameter $k>\rank(A)=10$, the distribution of this type of initial points is similar to the one in the first row of Figure~\ref{fig:svd}, which has a large gap of singular values. We make a test on different rank parameters $k$ chosen from the set $\{10,11,\dots,20\}$. The numerical results are presented in Figure~\ref{fig:RRAM-error}, and observations are summarized as follows.

\begin{figure}[h]
	\centering
	\subfigure[Relative gradient]
	{\includegraphics[width=.325\textwidth]
		{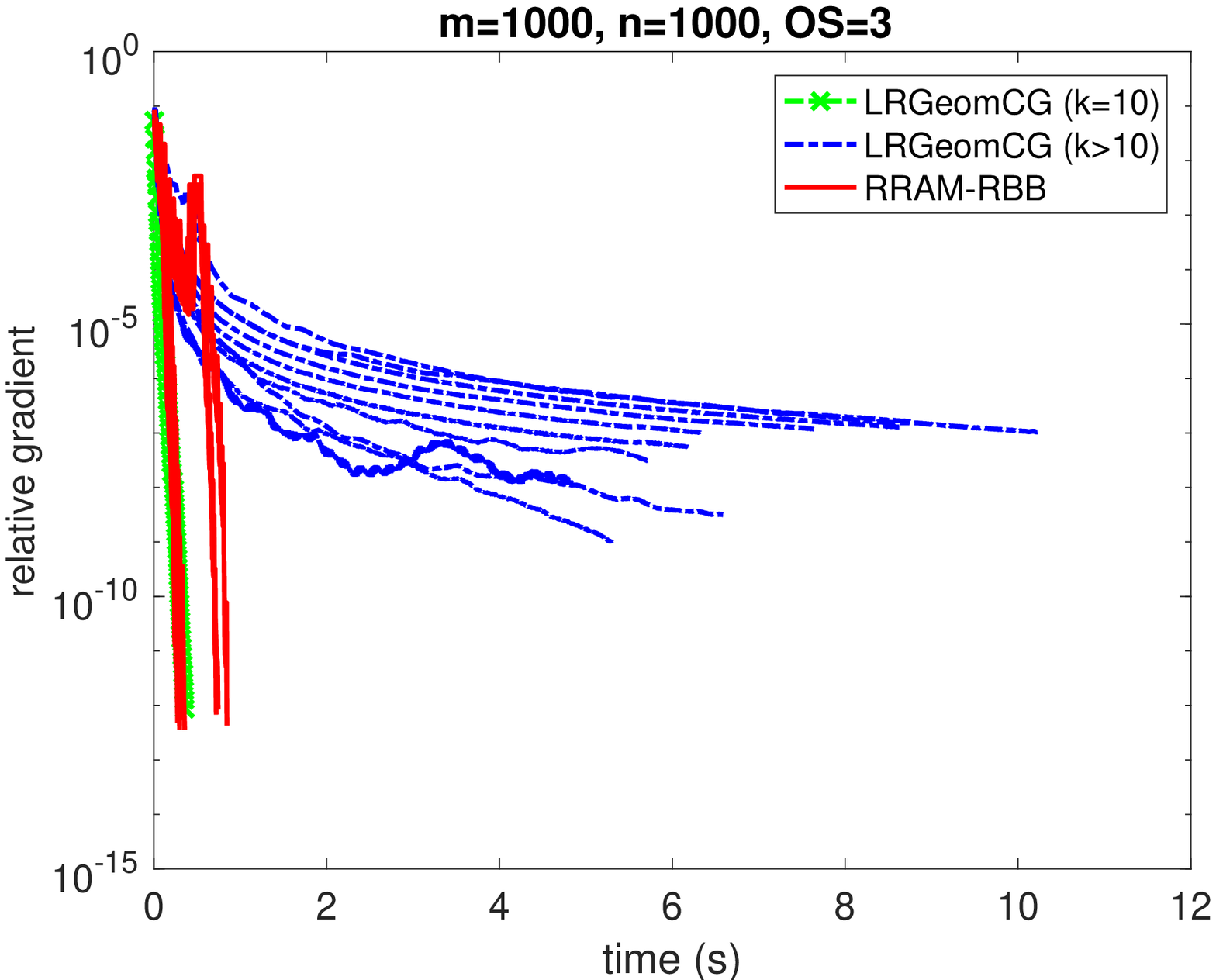}}
	\subfigure[Relative residual]
	{\includegraphics[width=.325\textwidth]
		{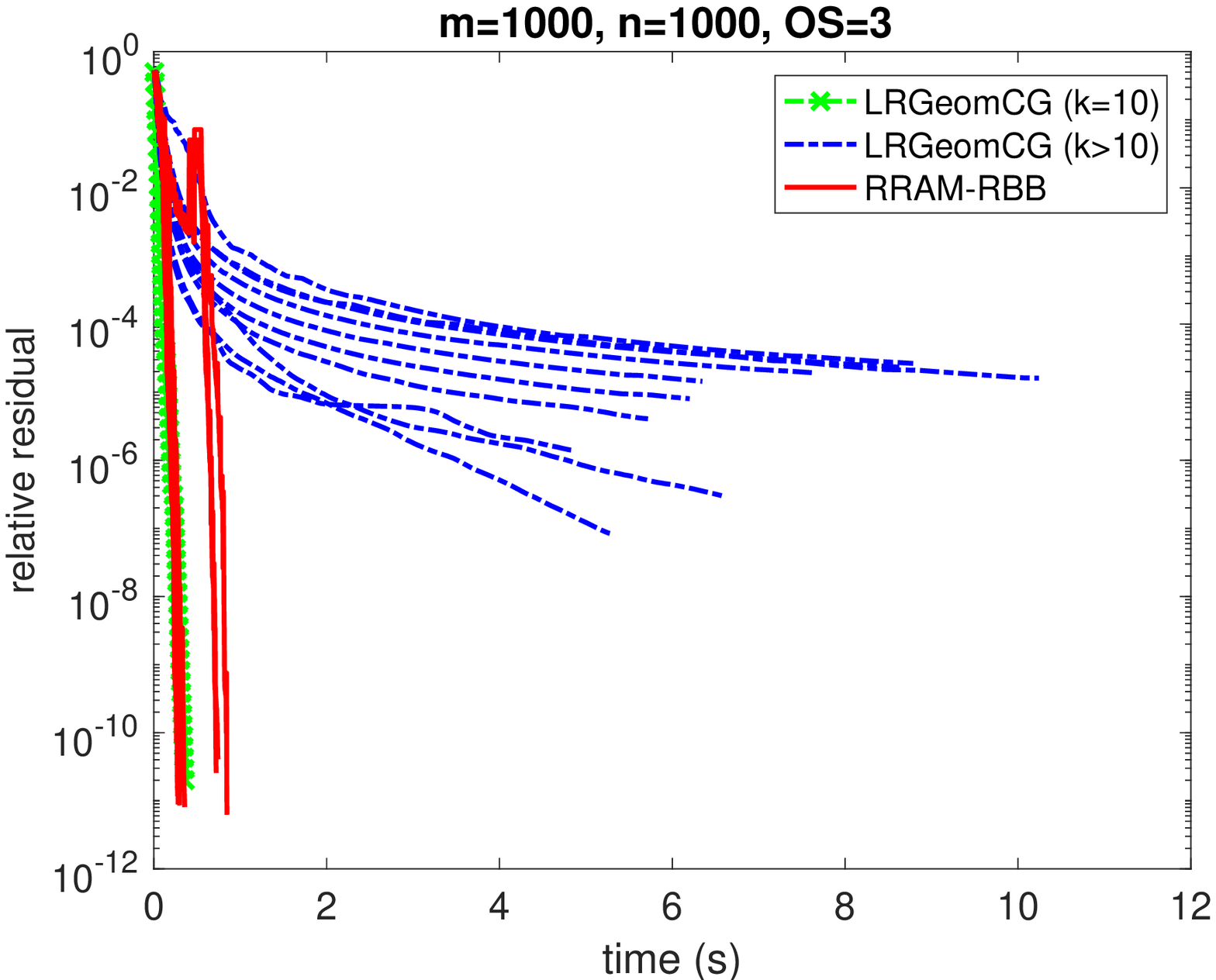}}
	\subfigure[Update rank]
	{\includegraphics[width=.31\textwidth]
		{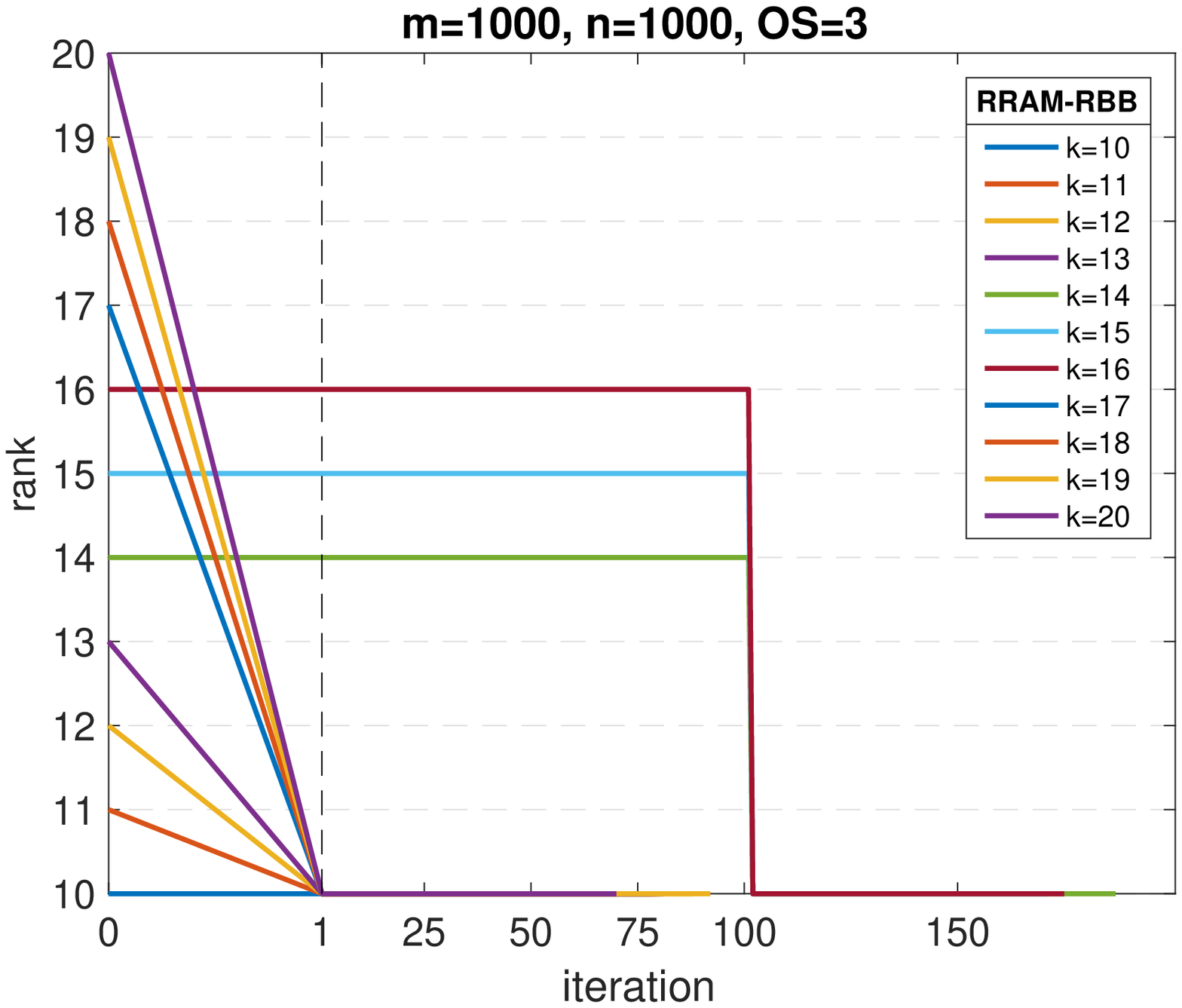}}
	\caption{A comparison with different rank parameters $k$. The initial point is generated by $X_0 = \projec_{\mathcal{M}_{k}}(\po(A))$.}
	\label{fig:RRAM-error}
\end{figure}

\begin{itemize}
	\item In Figure~\ref{fig:RRAM-error}(a)-(b), it is observed that for LRGeomCG, the best choice of $k$ is by far $k=10$, which is the true rank of data matrix $A$. It reveals that the performance of the fixed-rank optimization method LRGeomCG highly depends on the choice of rank parameter, while the proposed rank-adaptive method has comparable results among all choices.
	\item The update rank of RRAM is listed in Figure~\ref{fig:RRAM-error}(c). Notice that the rank reduction step is invoked in the initialization stage of Figure~\ref{fig:flowchart} for most choices of the initial rank, and it reduces the rank to 10. In the cases of $k=14,15,16$, although a initial rank reduction is not activated, the algorithm can detect the large gap of singular values when the first call of the fixed-rank method (Algorithm~\ref{alg:non-monotone gradient}) terminates (at its maximum iteration number 100) and reduces $k$ to the true rank~10.
	\item It is worth mentioning that in Figure~\ref{fig:RRAM-error}, when a rank reduction step completes, it often increases the function value at the very beginning, but the algorithm quickly converges once the true rank is detected.
\end{itemize}

Another class of initial points is randomly generated by a low-rank matrix $LR\zz$ that has rank~$k$. It has a uniform singular values distribution that is the same as the second row of Figure~\ref{fig:svd}. Similarly, we compare RRAM-RBB with LRGeomCG on the problems with different rank parameters, and the results are reported in Figure~\ref{fig:RRAM-error-random}. We observe that RRAM-RBB can reduce the rank among all choices of $k>10$ even when the singular values of {the} initial point do not have a large gap. Note that in the cases of $k=15$ and $18$, the first fixed-rank optimization stops with the iteration number less than $100$ since the relative change is achieved.

\begin{figure}[h]
	\centering
	\subfigure[Relative gradient]
	{\includegraphics[width=.325\textwidth]
		{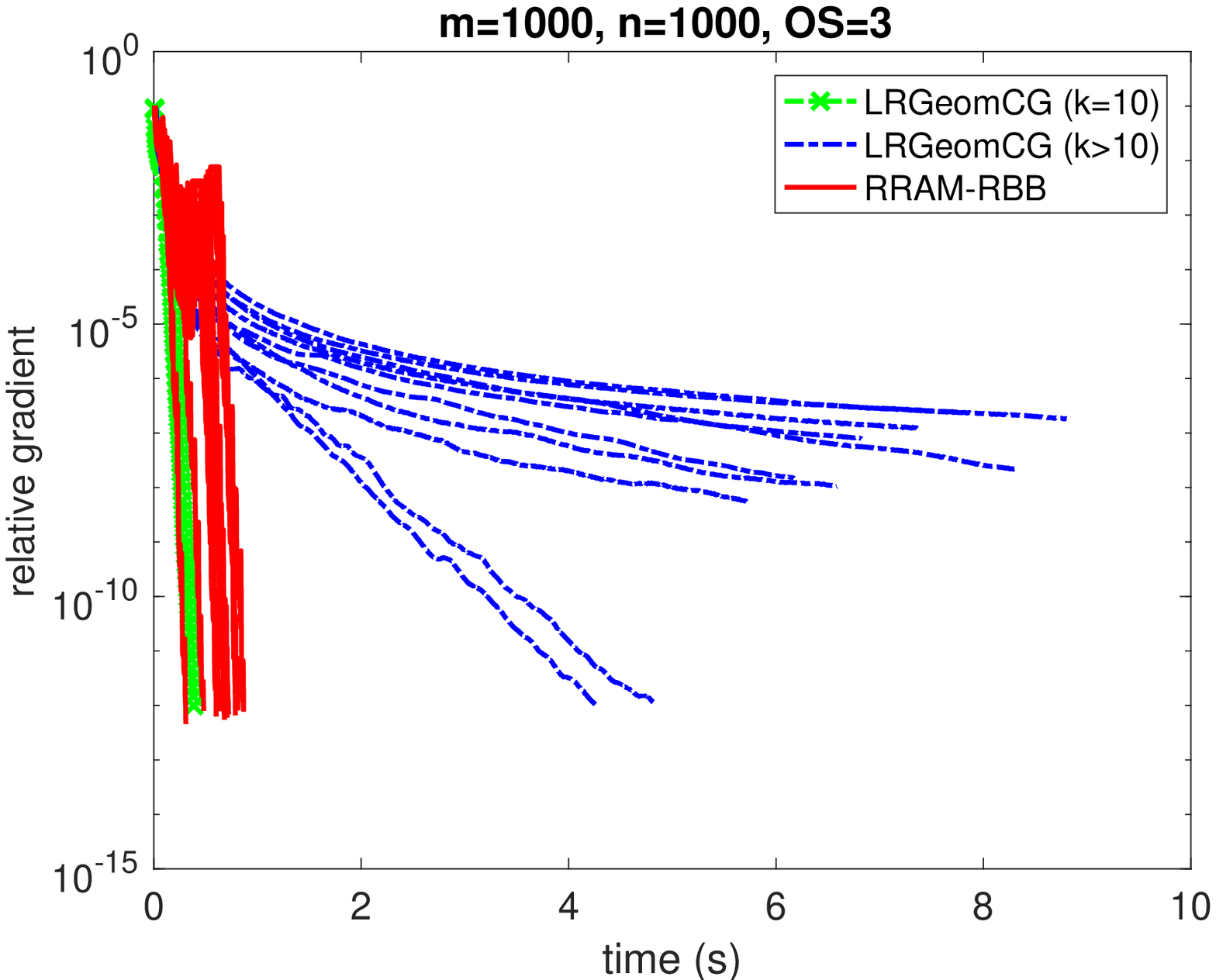}}
	\subfigure[Relative residual]
	{\includegraphics[width=.325\textwidth]
		{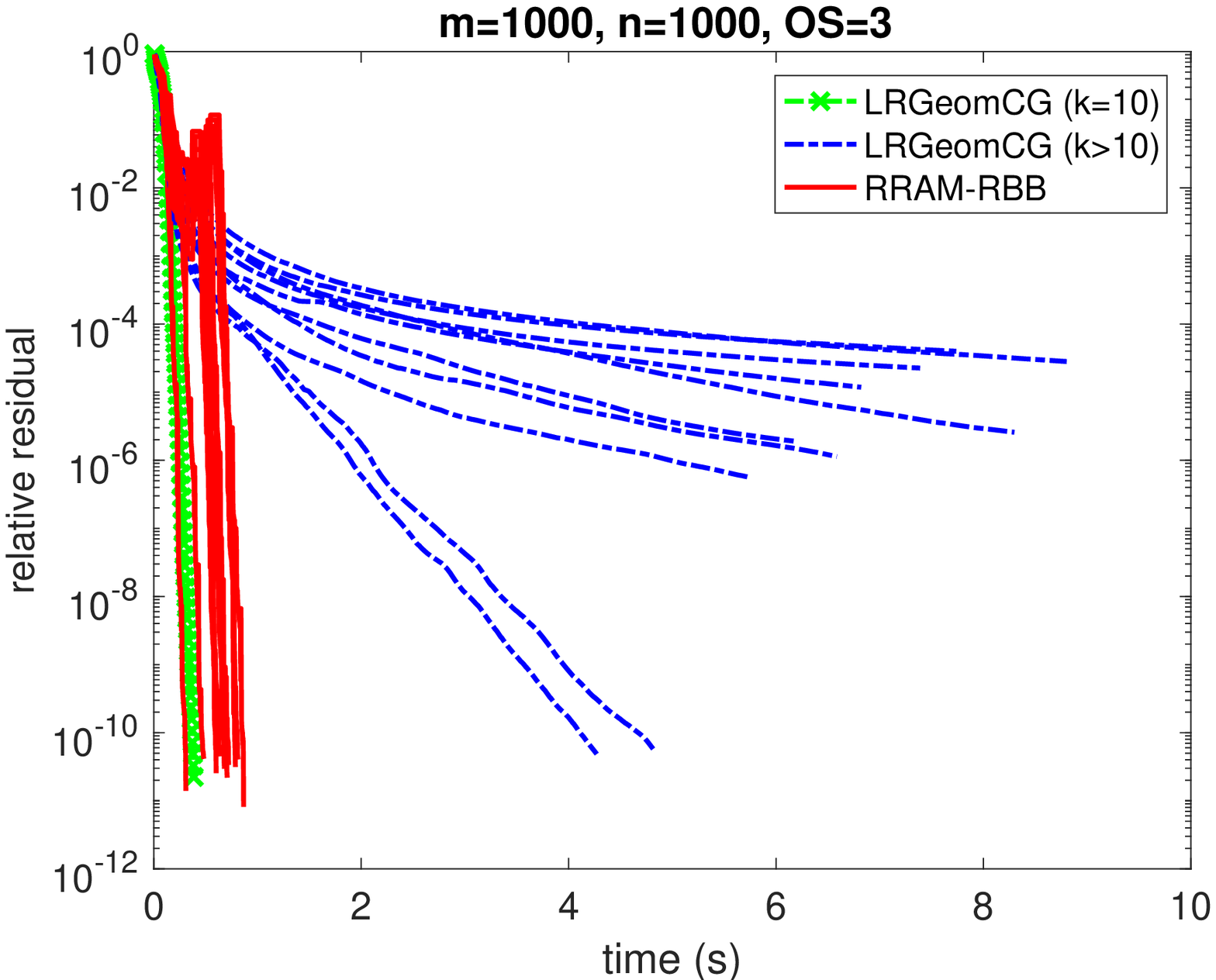}}
	\subfigure[Update rank]
	{\includegraphics[width=.31\textwidth]
		{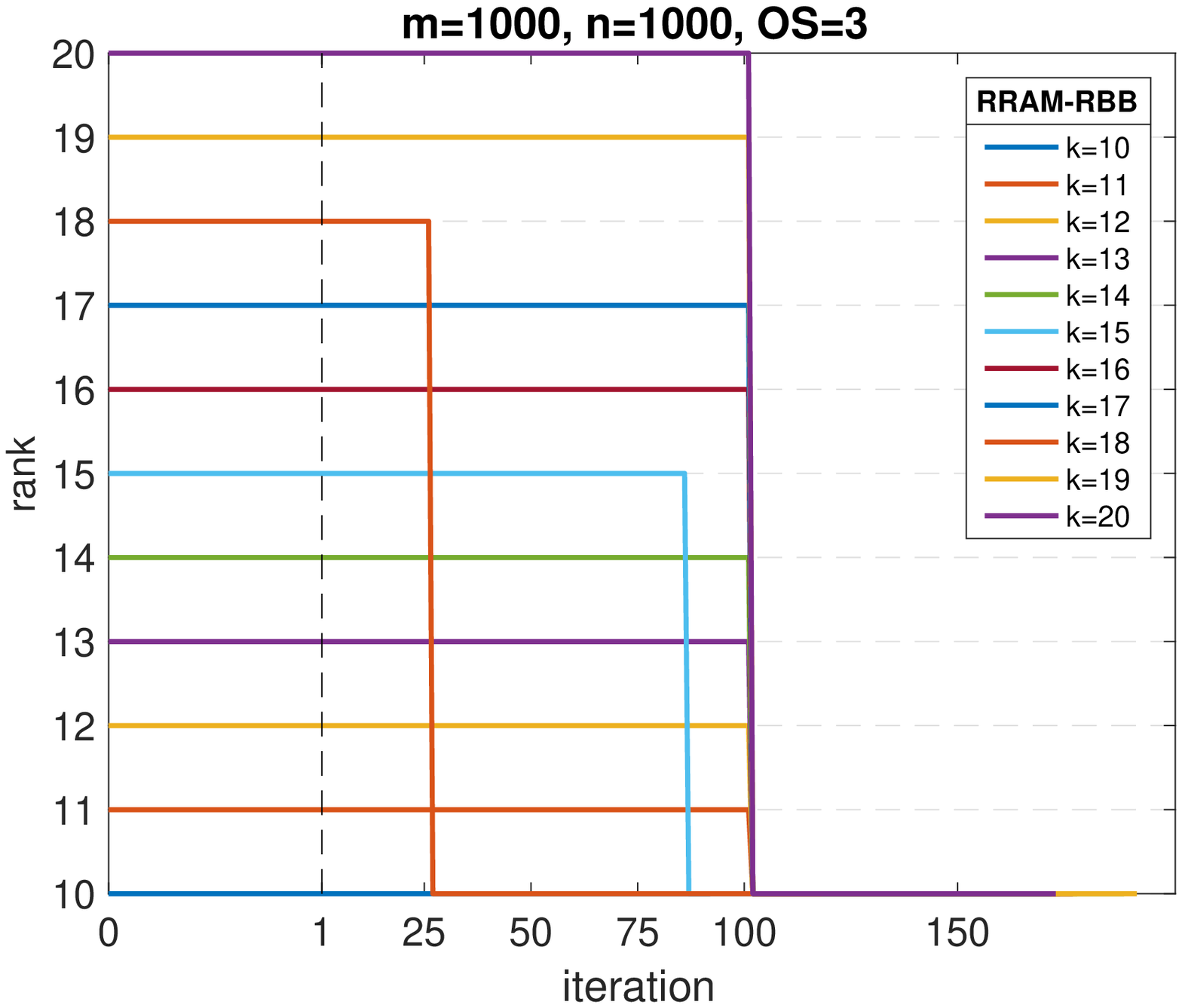}}
	\caption{A comparison with different rank parameters $k$. The initial point is randomly generated.}
	\label{fig:RRAM-error-random}
\end{figure}



%

\subsection{Comparison on the rank increase}\label{subsec:rank-increase-numerical}
In this subsection, we consider a class of problems {for which} the data matrix $A$ is ill-conditioned. This type of problem has been numerically studied in~\cite{Uschmajew_V:2015}. Specifically, we construct $$A=U\diag(1,10^{-1},\dots,10^{-{r+1}})V\zz,$$
where $U\in\R^{m\times r}$ and $V\in\R^{n\times r}$. Note that $A$ has exponentially decaying singular values. We generate the problem with $m=n=1000$, $k=r=20$ and $\mathrm{OS}=3$. The initial point is generated by~\eqref{eq:initialization}. We choose the rank increase parameter $\epsilon=2$ such that RRAM-RBB is prone to increase the rank. The tolerance parameter $\epsilon_g$ is set to $10^{-15}$.

We test on three different settings: (I) the maximum iteration number $j_{\max}$ for the fixed-rank optimization is set to 5, and the rank increase number $l=1$; (II)~$j_{\max}=100$ and $l=1$;  (III) $j_{\max}=20$ and $l=2$. Figure~\ref{fig:RRAM-increase} reports the evolution of errors and the update rank of RRAM-RBB. The observations are as follows.
\begin{itemize}
	\item In this ill-conditioned problem, RRAM-RBB performs better than the fixed-rank optimization method LRGeomCG ($k=20$). In addition, we observe that the rank reduction step is invoked at the initial point for three settings, and RRAM-RBB increases the rank by a number $l$ after each fixed-rank optimization.
	\item Note that the oscillation  of relative gradient in RRAM-RBB stems from the rank increase step. From the first two columns of Figure~\ref{fig:RRAM-increase}, it is observed that if the fixed-rank problem is inexactly solved ($j_{\max}=5$), the performance of RRAM-RBB is still comparable with the ``exactly" solved algorithm ($j_{\max}=100$). 
\end{itemize}

\begin{figure}[h]
	\centering
	{\includegraphics[width=.325\textwidth]
		{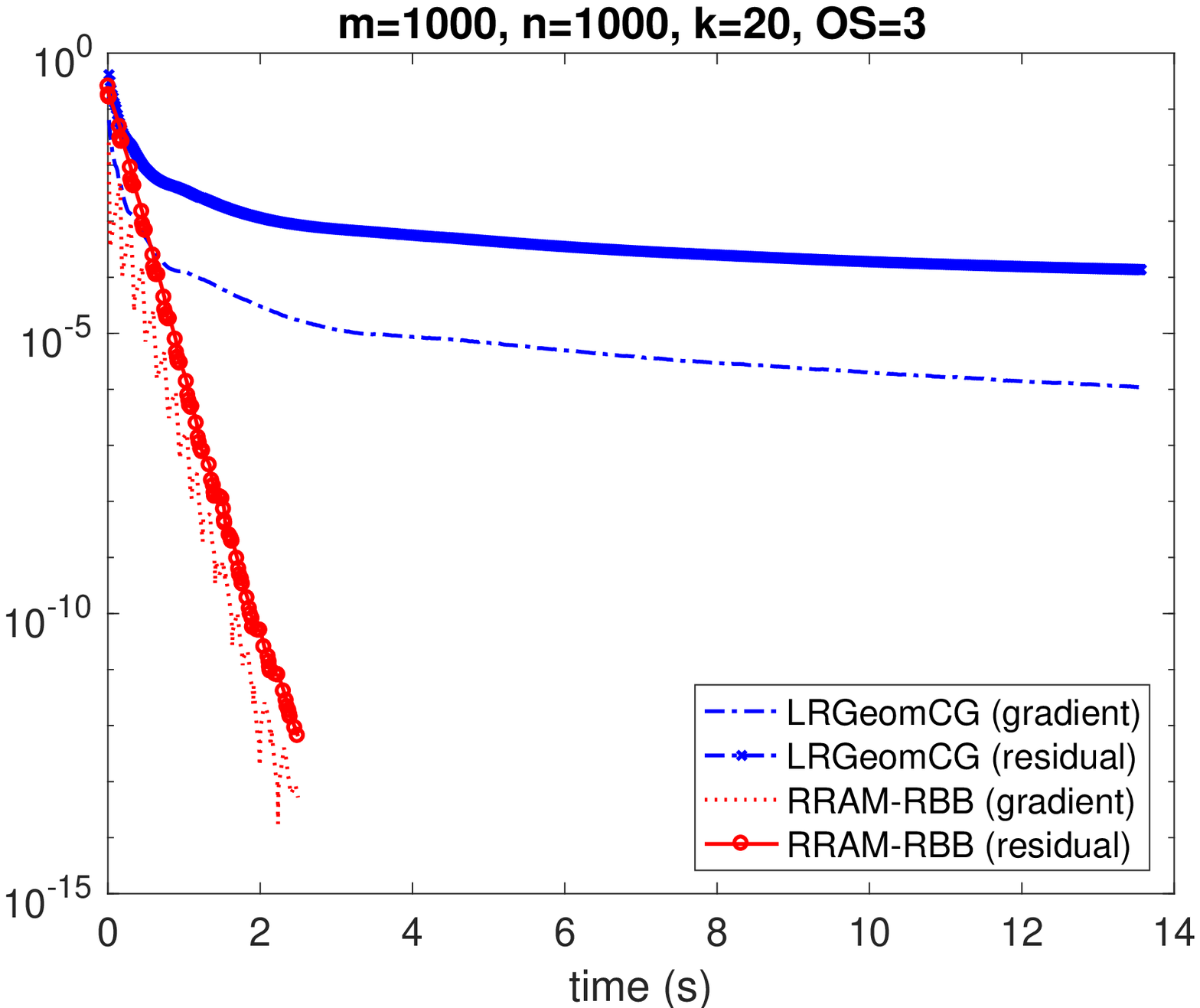}}
	{\includegraphics[width=.325\textwidth]
		{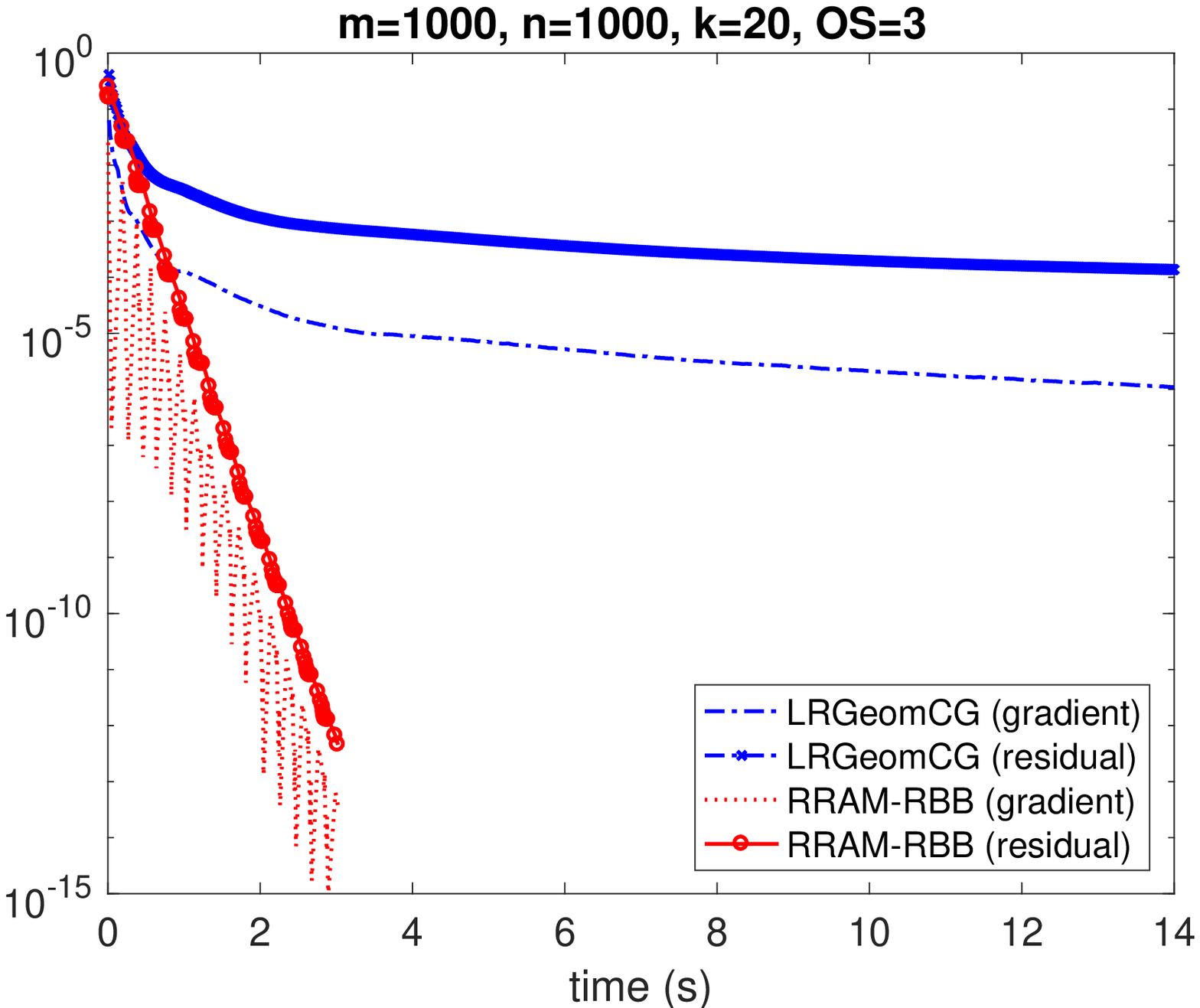}}	
	{\includegraphics[width=.325\textwidth]
		{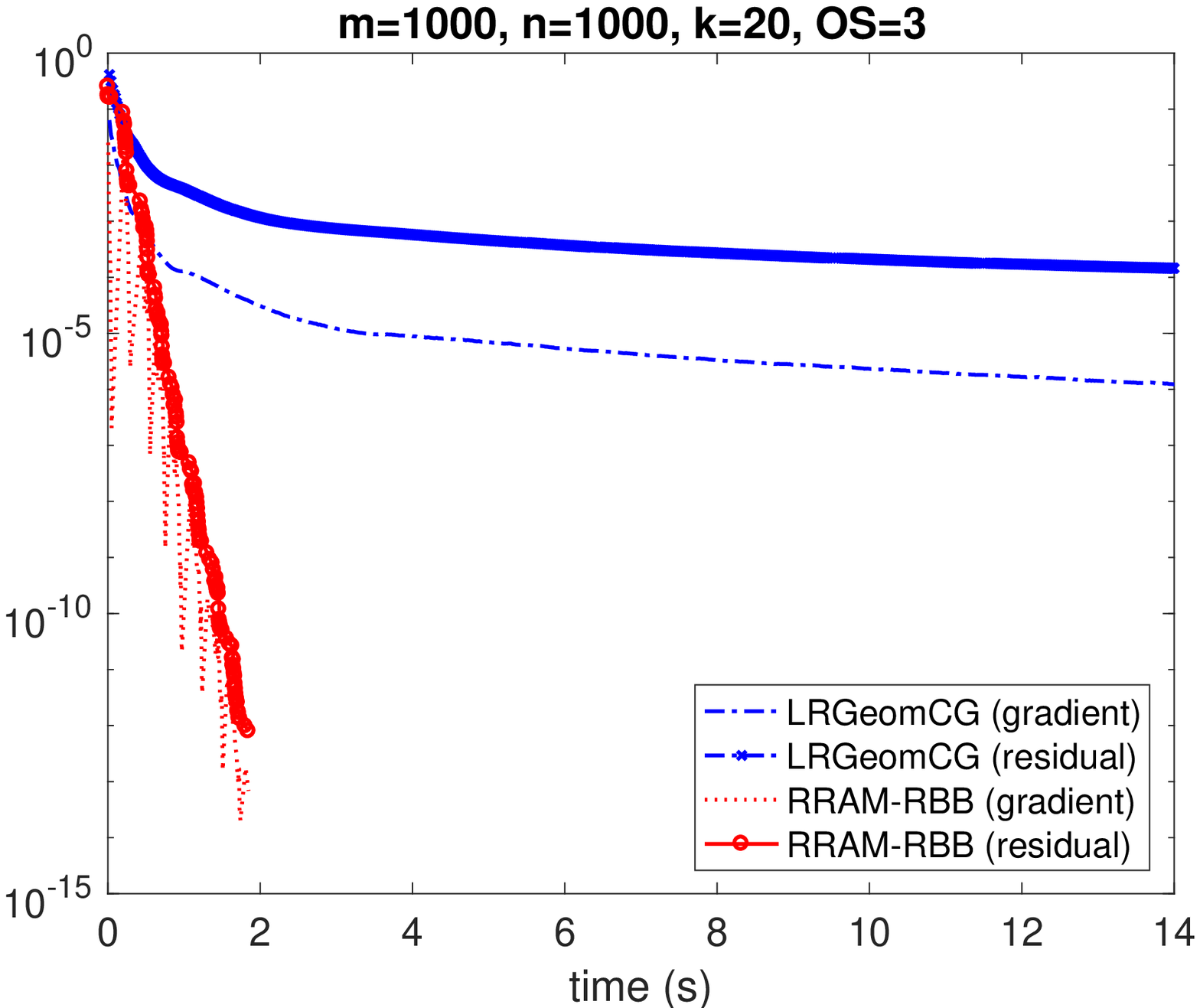}}
	\subfigure[$j_{\max}=5$ and $l=1$]
	{\includegraphics[width=.325\textwidth]
		{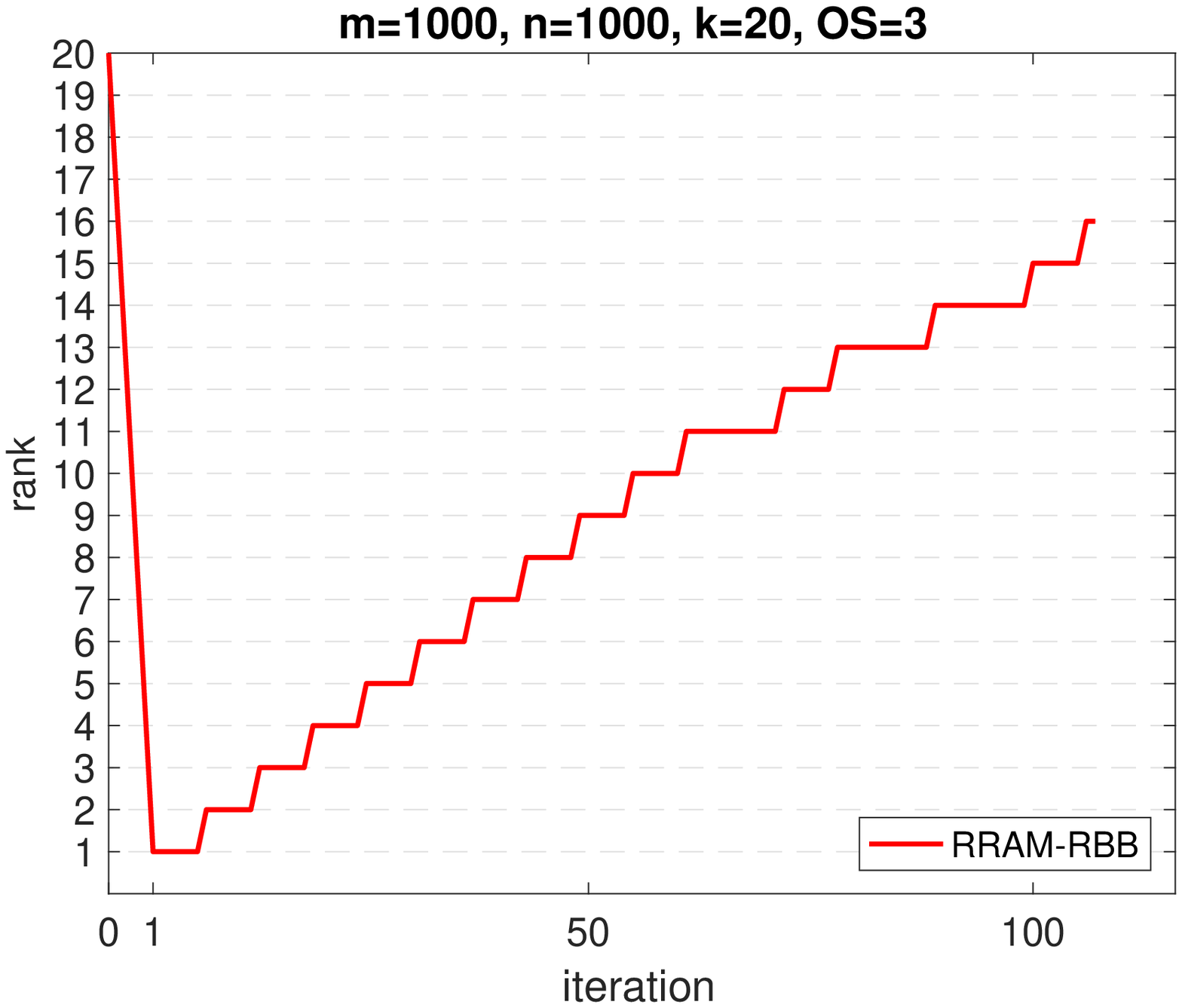}}
	\subfigure[$j_{\max}=100$ and $l=1$]
	{\includegraphics[width=.325\textwidth]
		{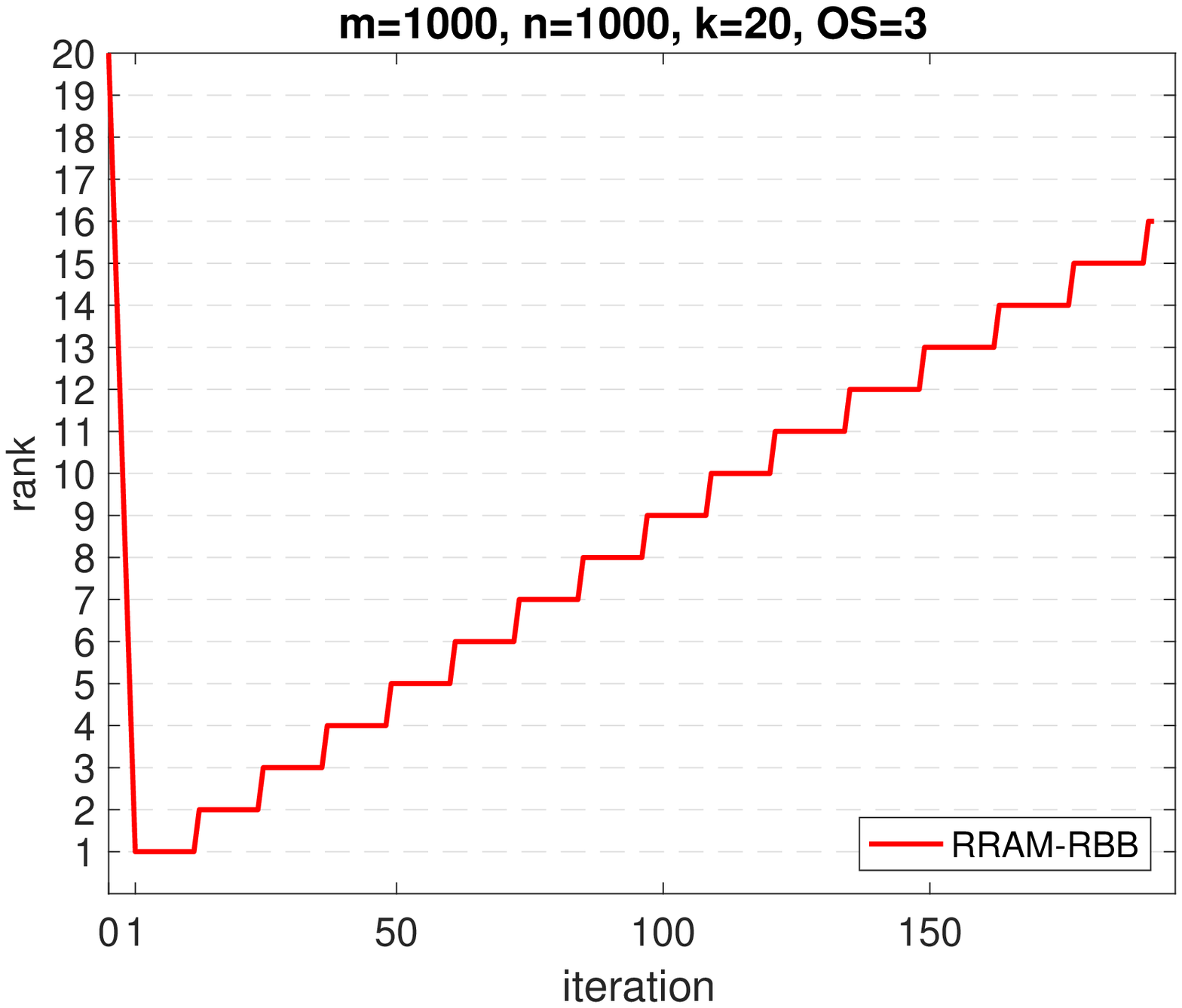}}
	\subfigure[$j_{\max}=20$ and $l=2$]
	{\includegraphics[width=.325\textwidth]
		{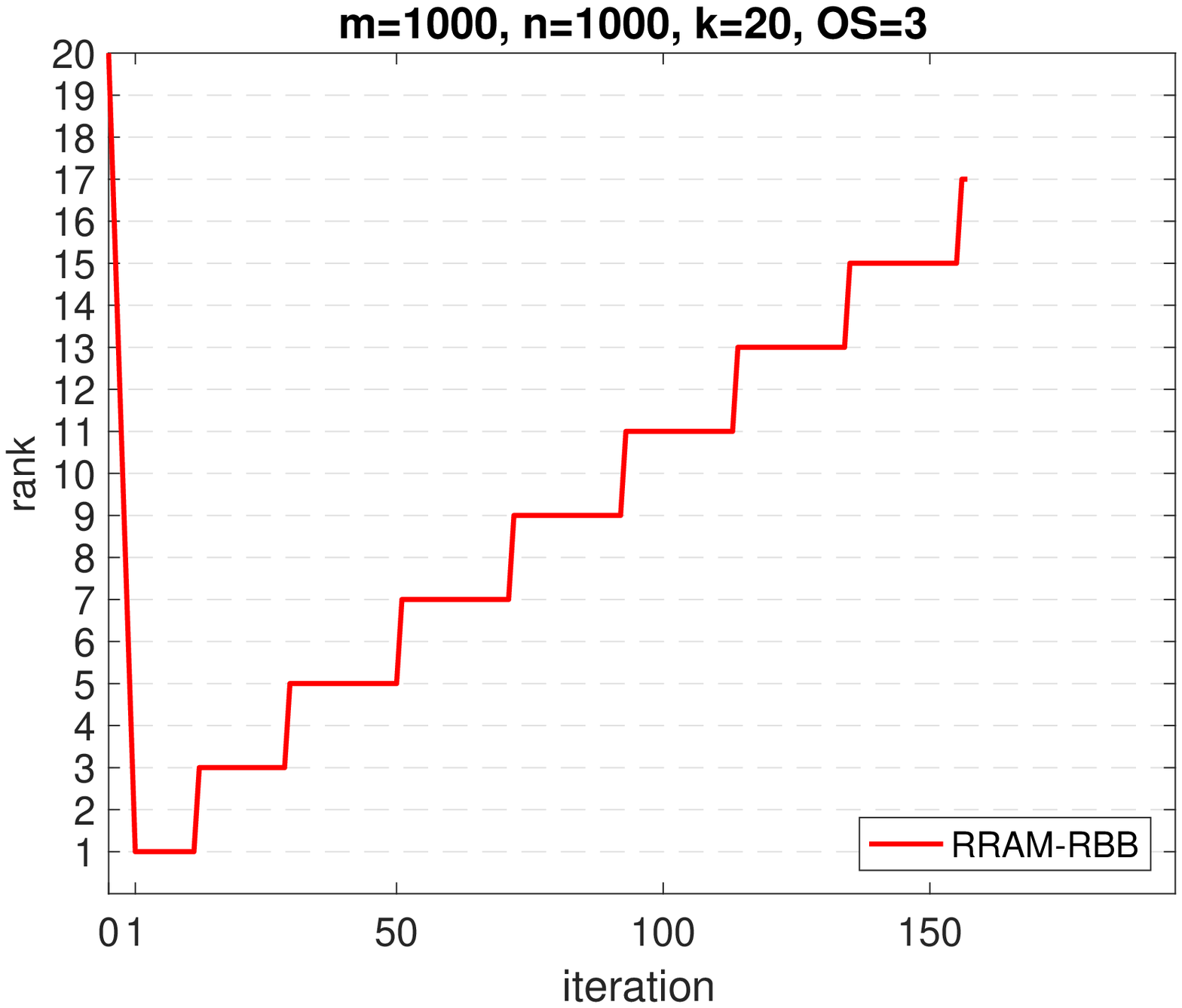}}
	\caption{A comparison on the rank increase with three settings. First row: evolution of errors. Second row: update rank.
	} 
	\label{fig:RRAM-increase}
\end{figure}

%
%

\subsection{Ablation comparison on the proposed rank-adaptive mechanism}\label{subsec:ablation-numerical}
In this subsection, 
we produce an ablation study by incorporating the framework of RRAM (Figure~\ref{fig:flowchart}) into the fixed-rank optimization LRGeomCG~\cite{vandereycken2013low} (Riemannian CG method). The resulting algorithm is called \texttt{RRAM-RCG}. Note that RRAM-RCG and RRAM-RBB differ only for the inner iteration.

In the first test, we compare RRAM-RCG with RRAM-RBB on problem instances generated as in subsection~\ref{subsec:rank-reduction-numerical}, with the random initial guess described therein. Specifically, the problem is generated with $m=n=10000$, $k=15$, $r=10$ and $\mathrm{OS}=3$. For both fixed-rank methods (RBB and RCG), the maximum iteration number $j_{\max}$ is set to $100$. In Figure~\ref{fig:RRAM-ablation-reduction}, the numerical results illustrate that RCG still enjoys the benefit of the rank reduction step~\eqref{eq:gap_sv} that reduces the rank from 15 to the true rank 10. Moreover, it indicates that using RBB instead of RCG yields a considerable improvement on this problem instance. This observation can be explained by the comparison in subsection~\ref{subsec:fixed-rank-numerical}.

\begin{figure}[h]
	\centering
	\subfigure[Relative gradient and relative residual]
	{\includegraphics[width=.48\textwidth]
		{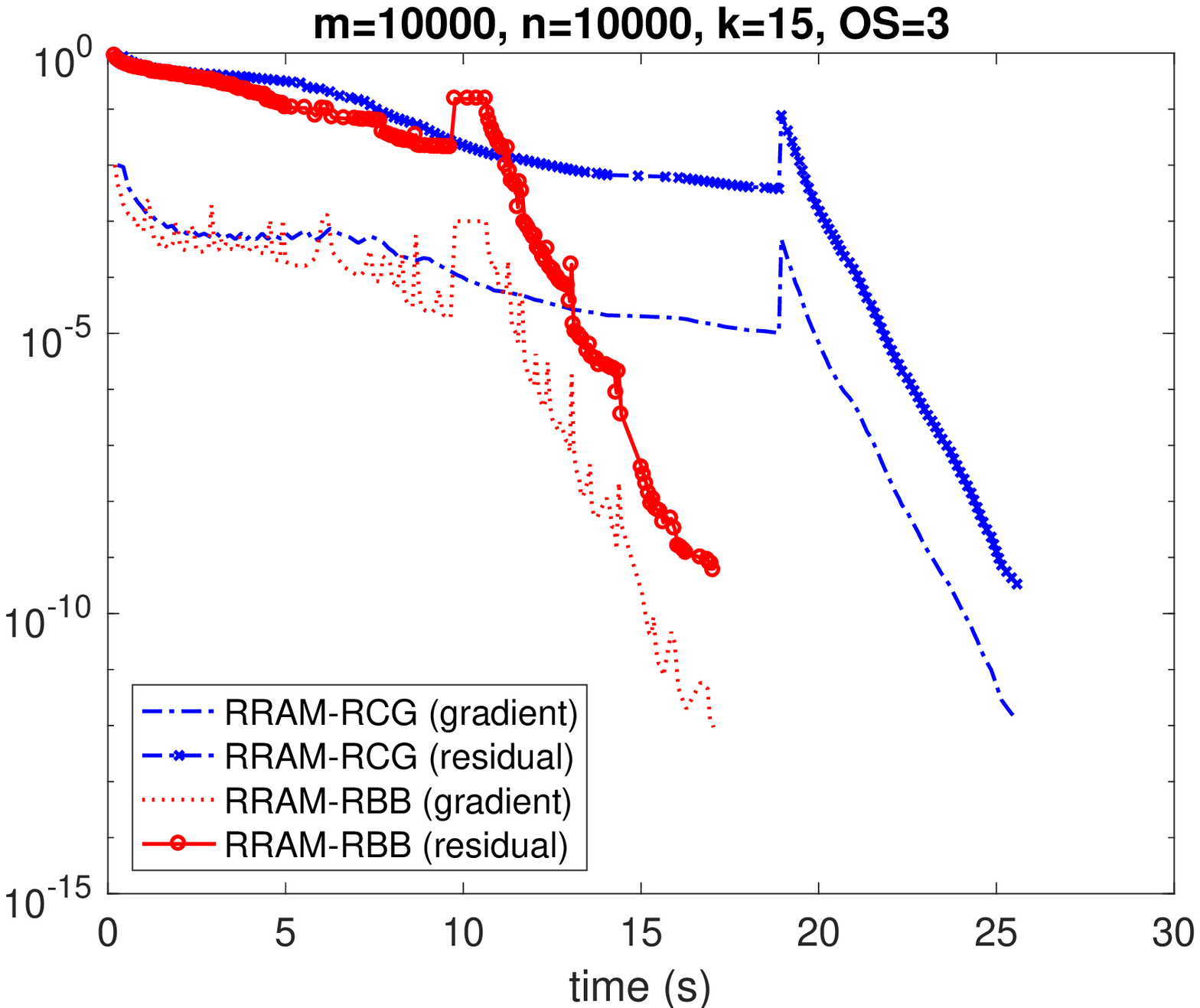}}
	\quad
	\subfigure[Update rank]
	{\includegraphics[width=.485\textwidth]
		{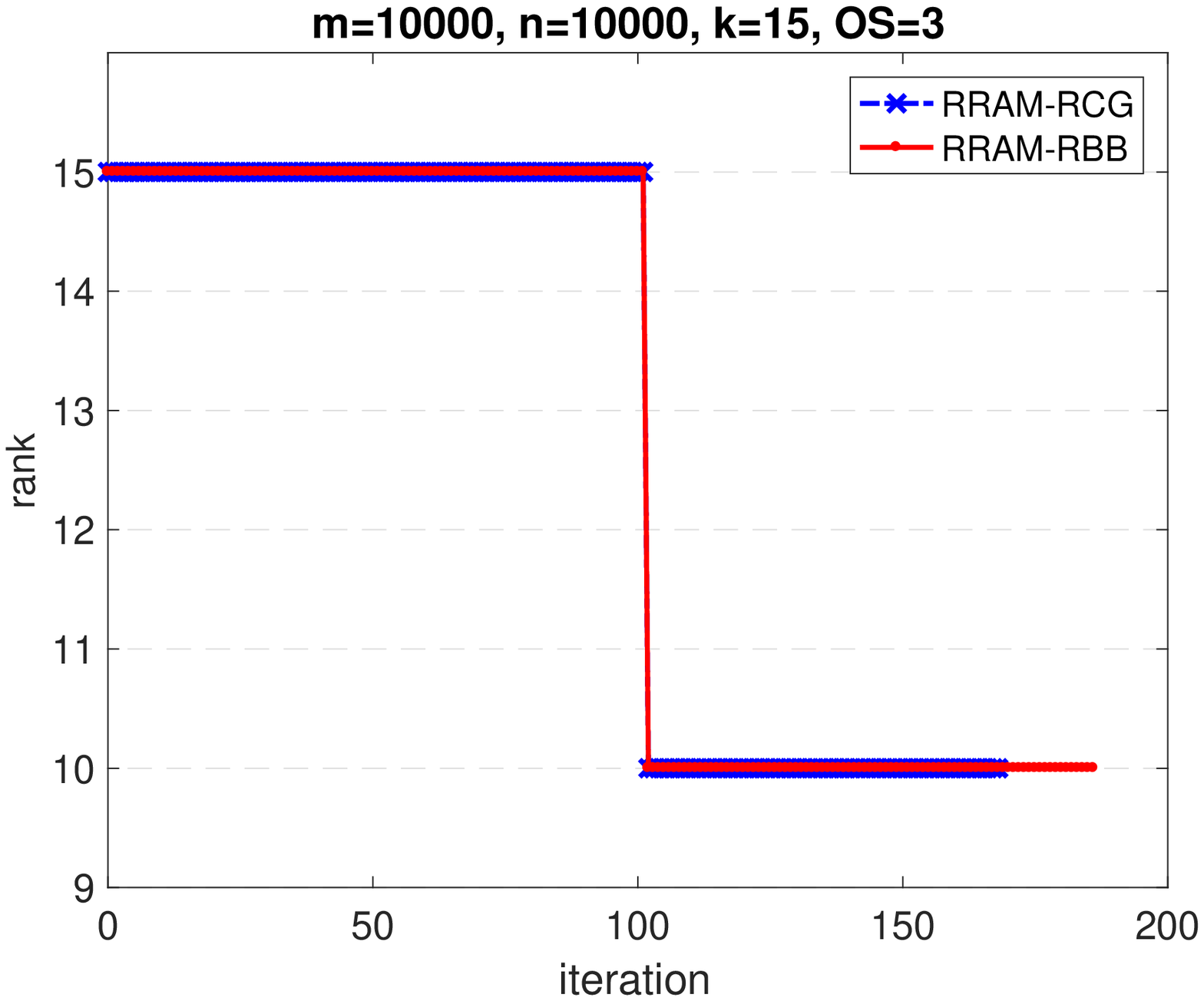}}
	\caption{An ablation comparison on the rank reduction.}
	\label{fig:RRAM-ablation-reduction}
\end{figure}

Another test is generated as in subsection~\ref{subsec:rank-increase-numerical} with $m=n=50000$, $k=r=20$ and $\mathrm{OS}=3$.  Figure~\ref{fig:RRAM-ablation} reports the performance comparison of RRAM-RBB and RRAM-RCG. It shows that the proposed rank increase strategy is also effective for RCG. Notice that the performance of RRAM-RBB is slightly better than RRAM-RCG in terms of time efficiency. 

\begin{figure}[h]
	\centering
	\subfigure[Relative gradient and relative residual]
	{\includegraphics[width=.48\textwidth]
		{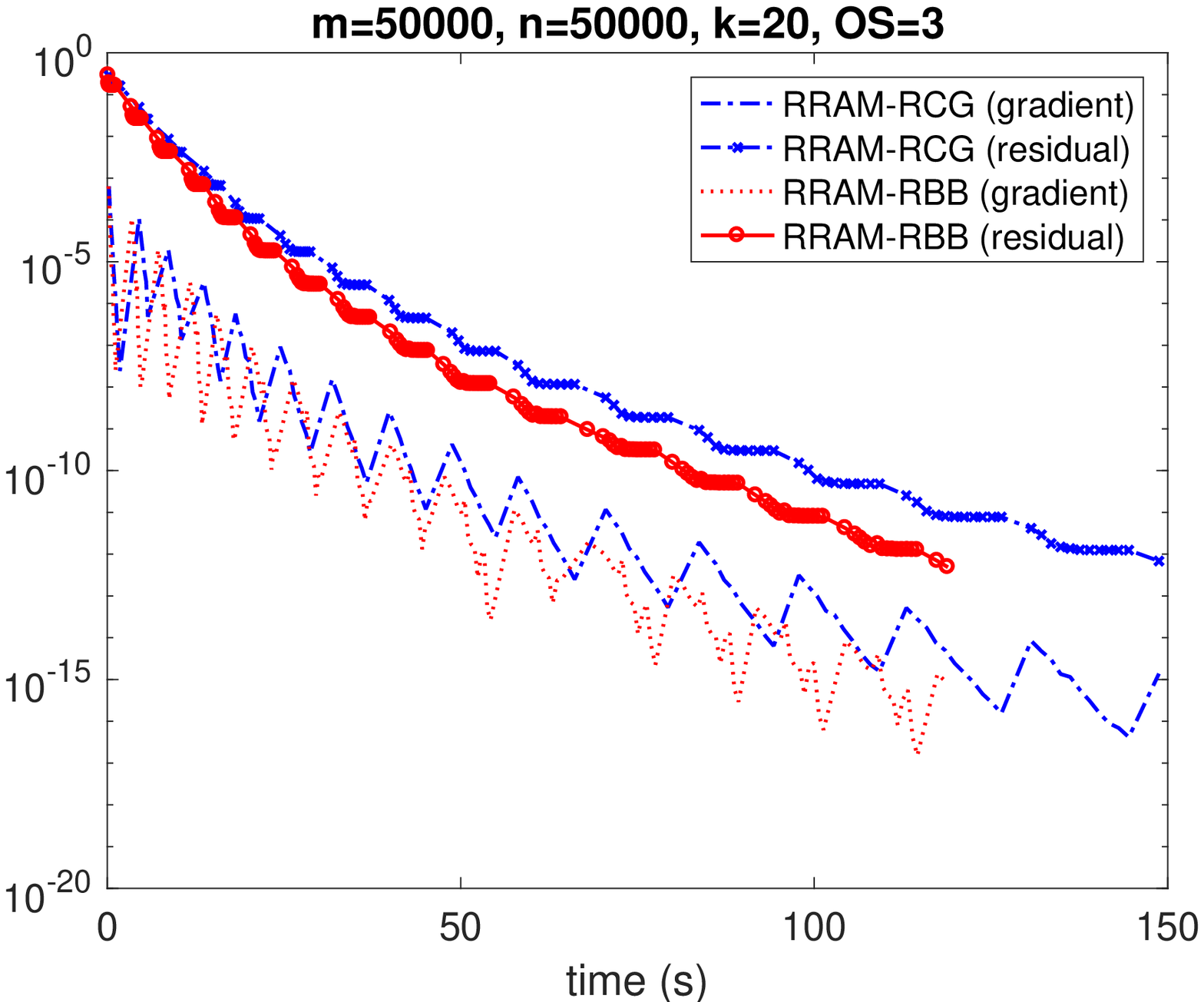}}
	\quad
	\subfigure[Update rank]
	{\includegraphics[width=.47\textwidth]
		{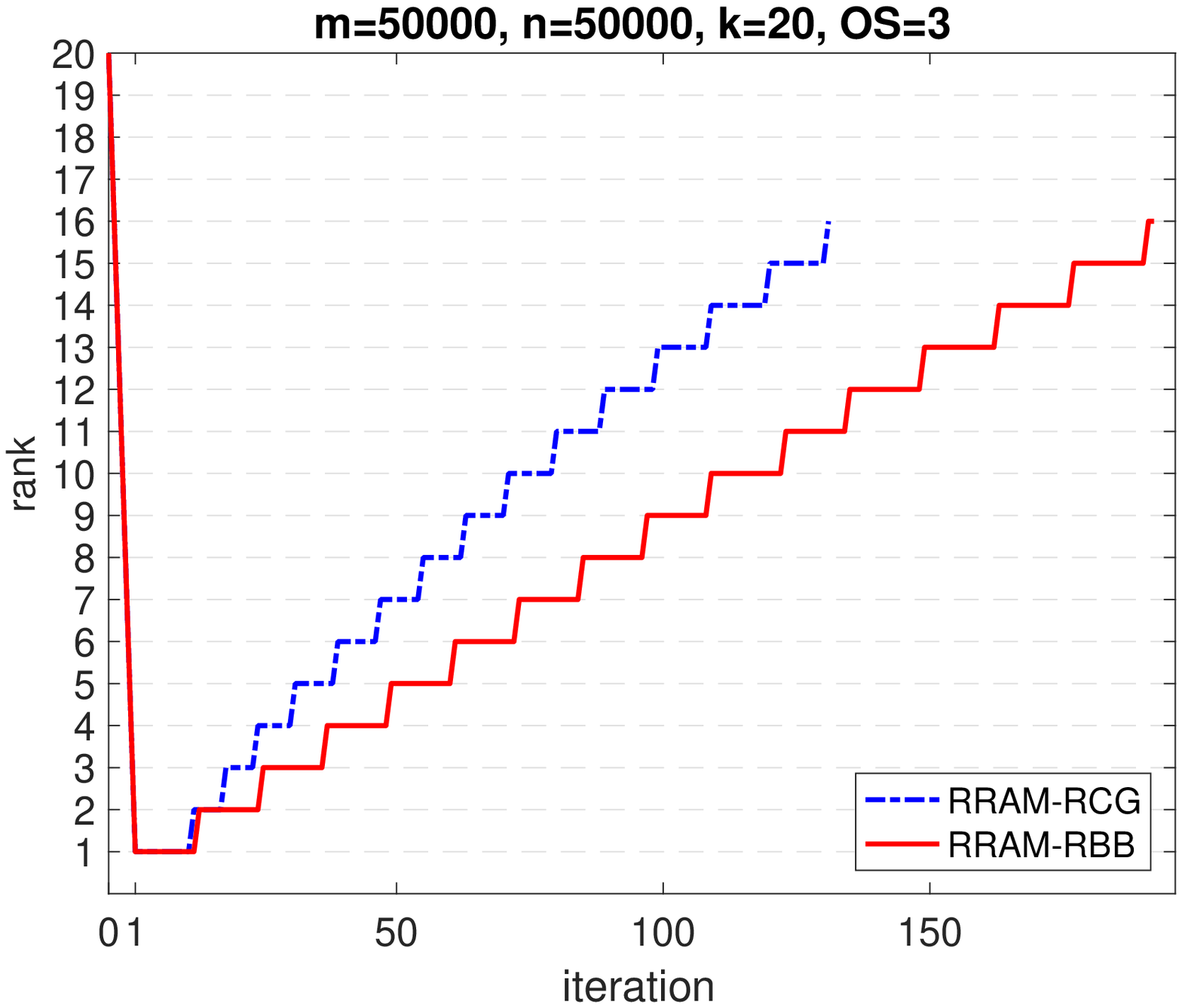}}
	\caption{An ablation comparison on the rank increase}
	\label{fig:RRAM-ablation}
\end{figure}

\subsection{Test on real-world datasets}
In this subsection, we evaluate the performance of RRAM on low-rank matrix completion with real-world datasets. The {MovieLens}\footnote{Available from \url{https://grouplens.org/datasets/movielens/}.} dataset contains movie rating data from users on different movies. In the following experiments, we choose the dataset {MovieLens100K} that consists of {$100000$} ratings from 943 users on 1682 movies, and {MovieLens 1M} that consists of one million movie ratings from 6040 users on 3952 movies. 

For comparison, we test RRAM-RBB with several state-of-the-art methods that particularly target low-rank matrix completion, namely, LRGeomCG\footref{footnote:RCG} \cite{vandereycken2013low}, NIHT\footnote{\label{footnote:ASD}Available from \url{http://www.sdspeople.fudan.edu.cn/weike/publications.html}.} and CGIHT\footref{footnote:ASD} \cite{wei2016guarantees}, ASD\footref{footnote:ASD} and ScaledASD\footref{footnote:ASD} \cite{tanner2016low}. Note that all these methods are based on the fixed-rank problem where the rank parameter has to be given a priori. For these two real-world datasets, we randomly choose 80\% of the known ratings as the training set and the rest as the test set. 
The rank parameter $k$ is set to 10 for all tested algorithms, and we terminate these algorithms once the time budget 
is reached or their own stopping criteria are achieved. 

\begin{figure}[htpb]
	\centering
	{\includegraphics[width=.48\textwidth]
		{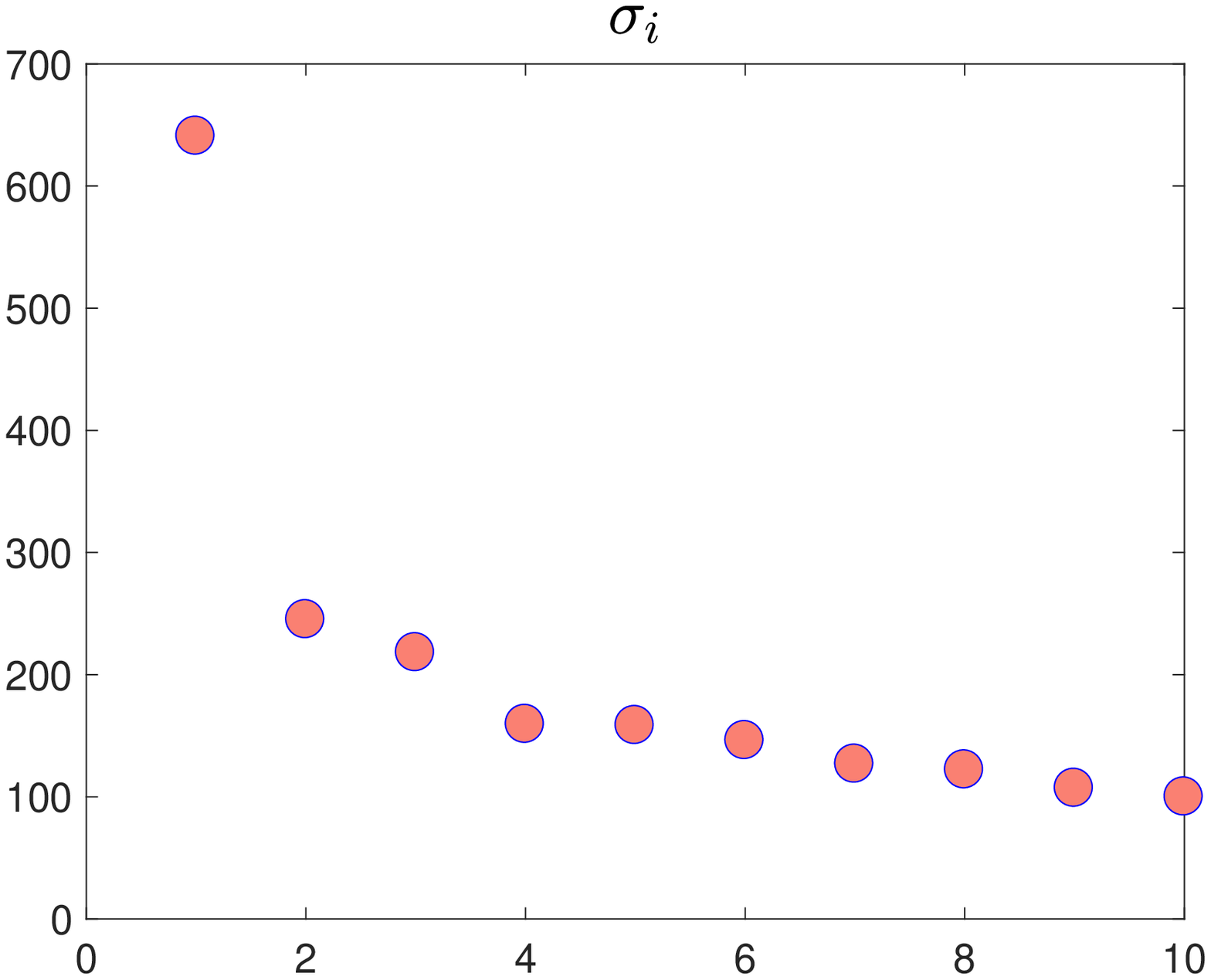}}
	\quad
	{\includegraphics[width=.47\textwidth]
		{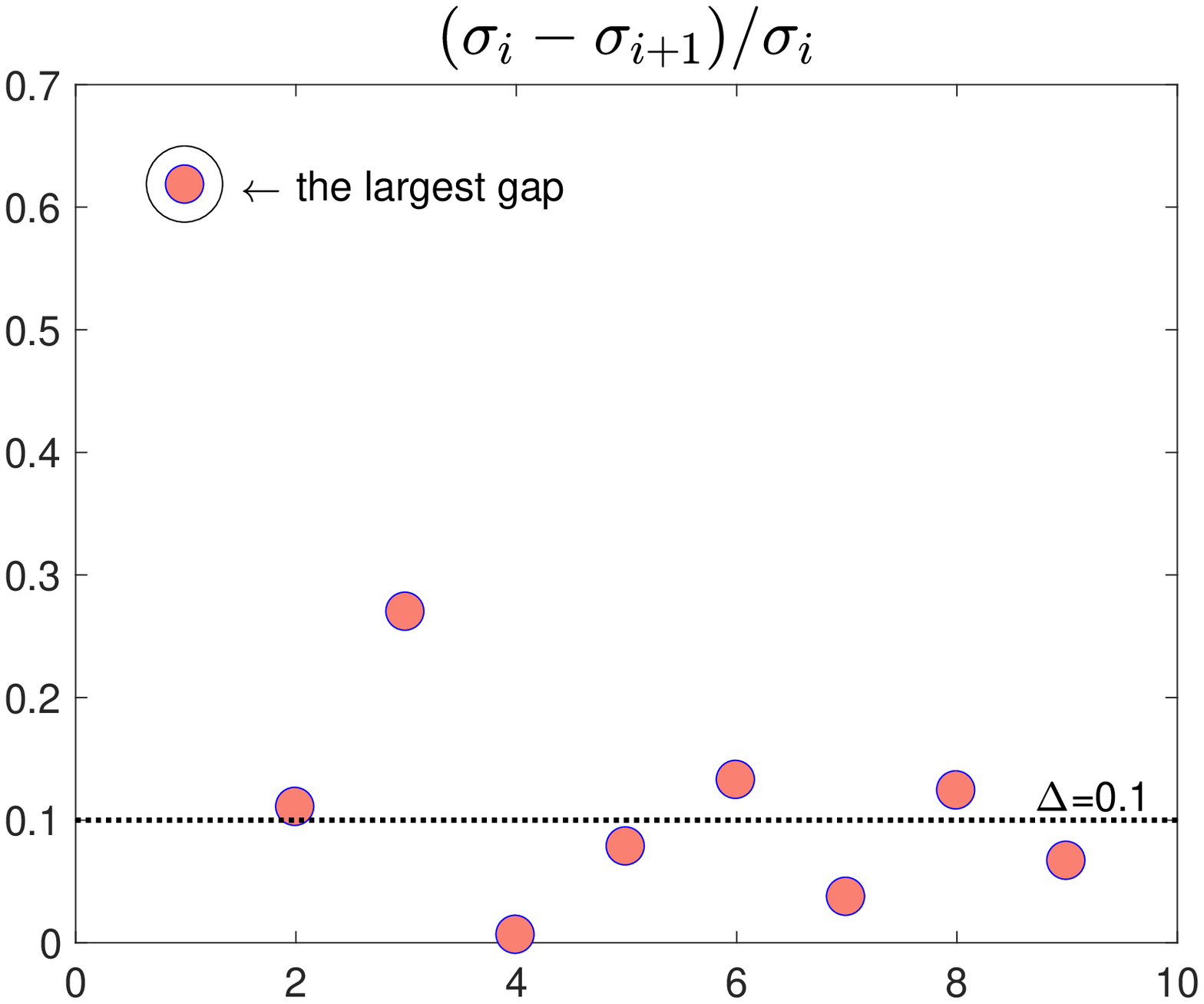}}
	
	\medskip
	{\includegraphics[width=.48\textwidth]
		{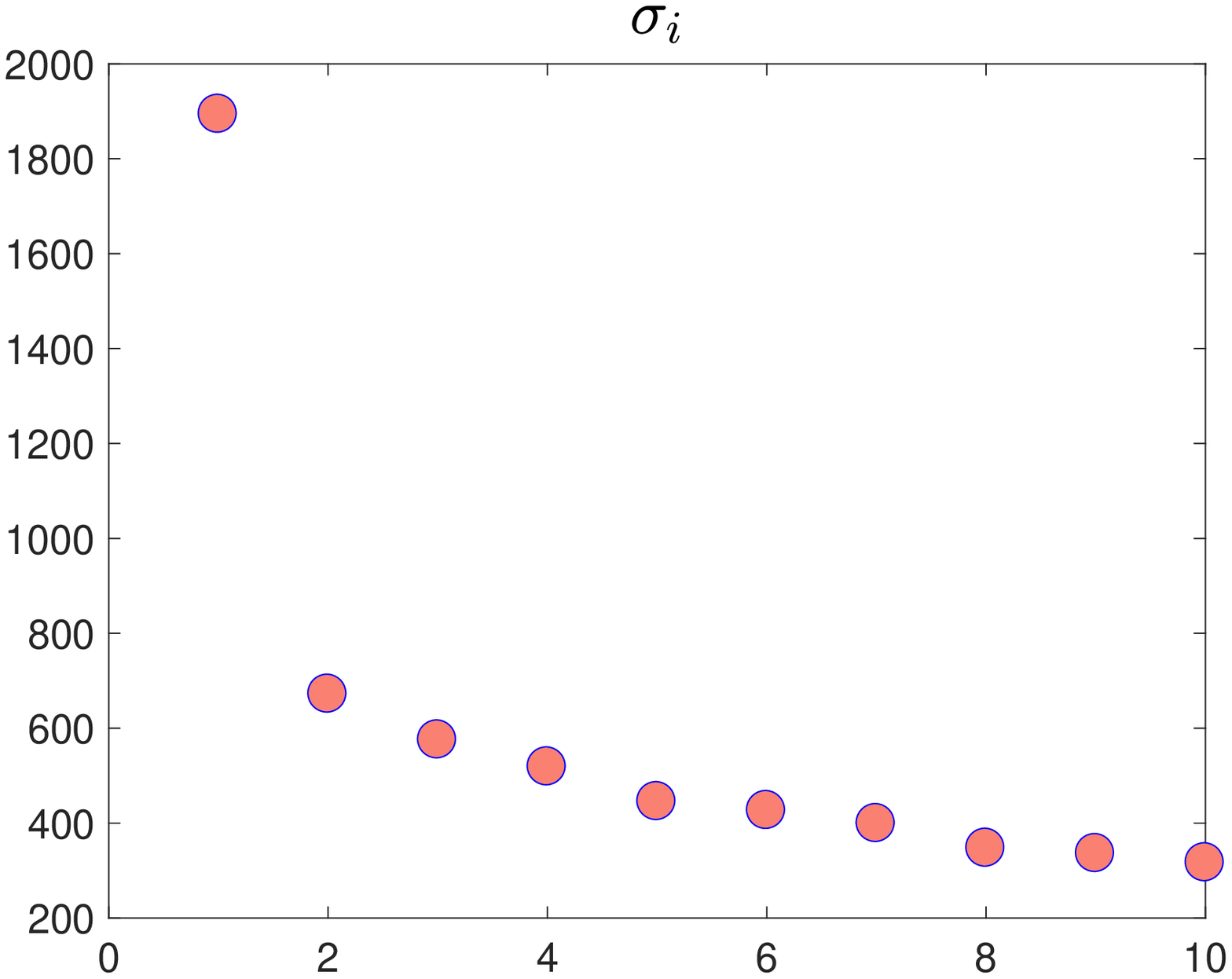}}
	\quad
	{\includegraphics[width=.465\textwidth]
		{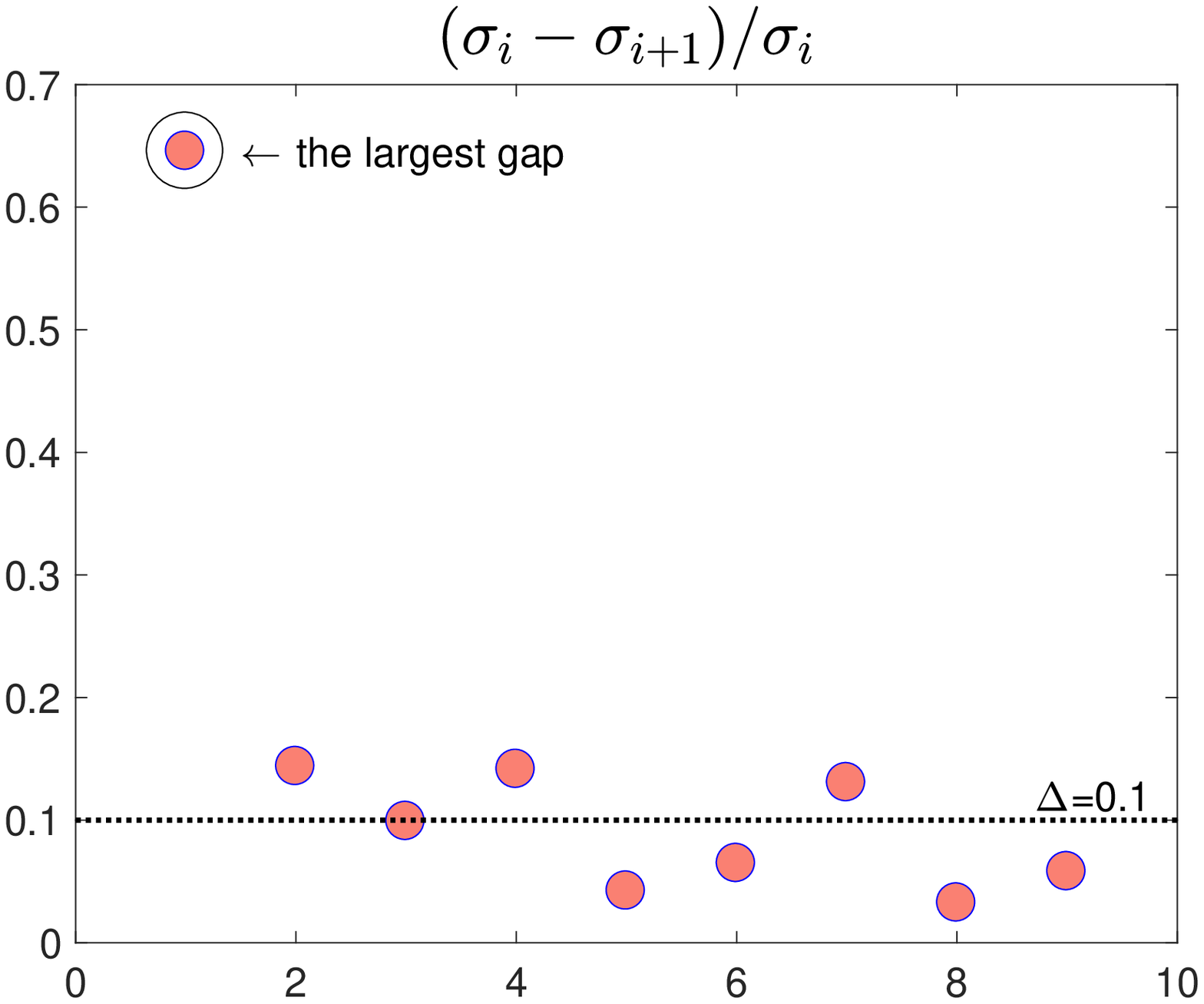}}
	\caption{Singular values of the initial points for the MovieLens dataset. First row: MovieLens100K. Second row: MovieLens 1M.  \label{fig:svd-movie}}
\end{figure}

Figure~\ref{fig:svd-movie} shows the singular values of the initial point~\eqref{eq:initialization}, namely the rank-10 approximation of the zero-filled MovieLens dataset. It is observed that the largest gap can be detected between the first two singular values by the rank reduction~\eqref{eq:gap_sv} for both examples. According to the rank-adaptive framework in Figure~\ref{fig:flowchart}, RRAM-RBB will thereby reduce the rank to one just after the initialization~\eqref{eq:initialization}, which explains the first rank reduction in the following figures for RRAM-RBB. 

The numerical results are illustrated in Figure~\ref{fig:RRAM-realdata}. Note that RRAM-RBB achieves the best final RMSE among all methods in the MovieLens100K dataset ($m=943$, $n=1682$), and is comparable with other algorithms in terms of time efficiency. The evolution of update rank of RRAM-RBB shows that
RRAM-RBB adaptively increases the rank and automatically finds a rank that is lower than the rank given to the other methods but with a smaller RMSE. In the larger dataset MovieLens 1M ($m=6040$, $n=3952$), RRAM-RBB still has a comparable RMSE. In summary, the rank-adaptive method accepts the flexible choices of rank parameter, and is able to search for a suitable rank.

\begin{figure}[h]
	\centering
	{\includegraphics[width=.48\textwidth]
		{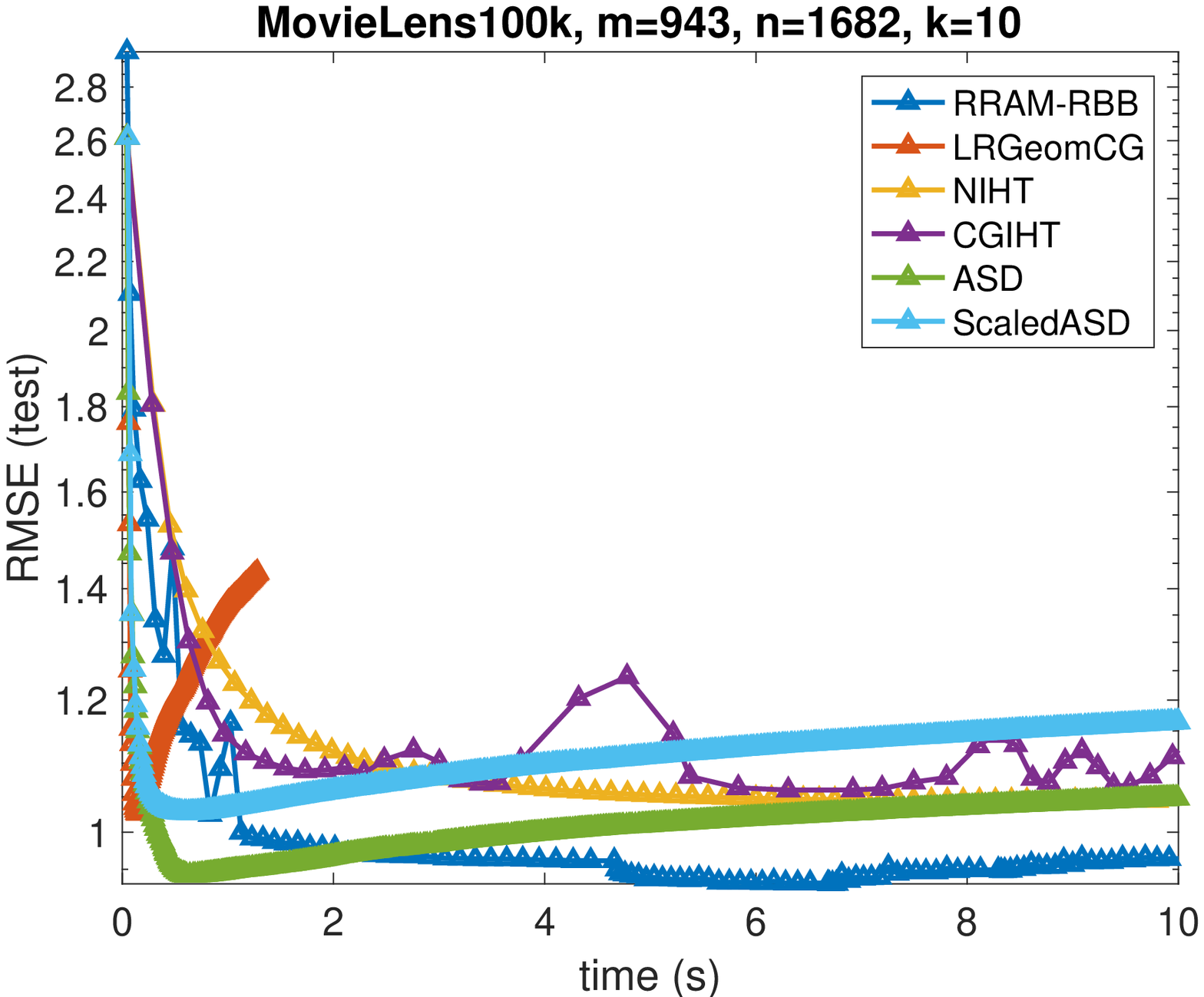}}
	\quad
	{\includegraphics[width=.47\textwidth]
		{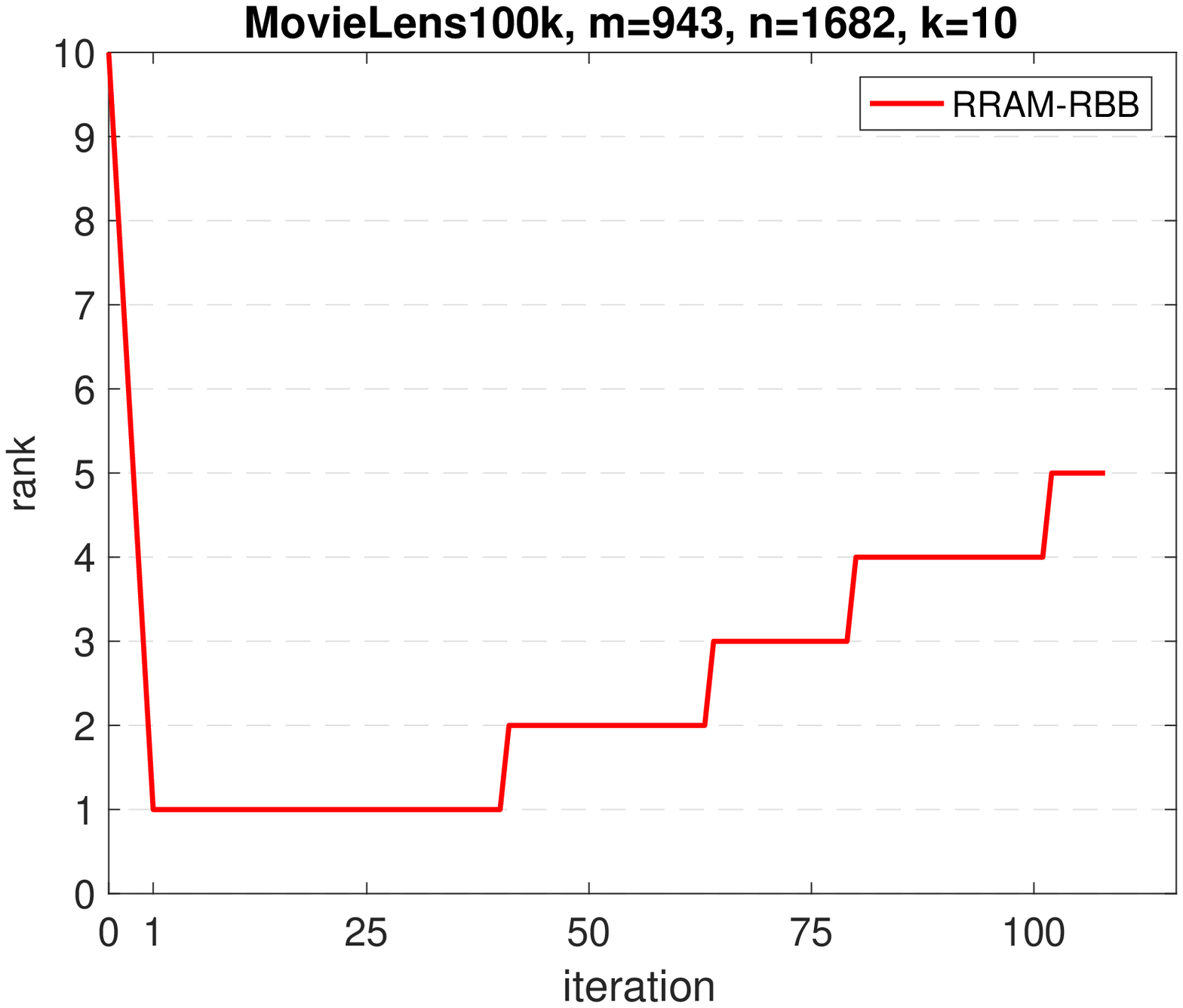}}
	\subfigure[RMSE (test)]
	{\includegraphics[width=.48\textwidth]
		{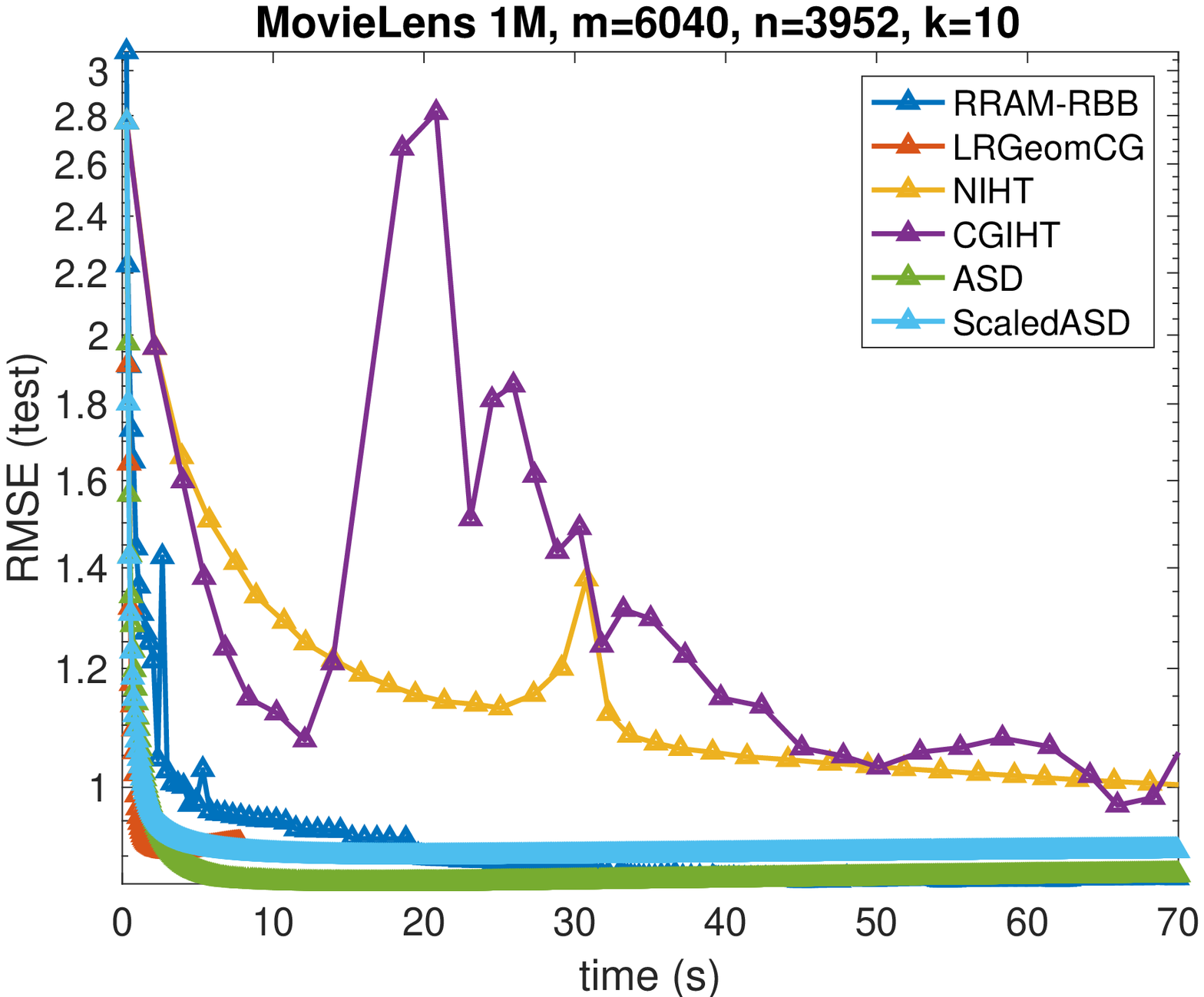}}
	\quad
	\subfigure[Update rank]
	{\includegraphics[width=.47\textwidth]
		{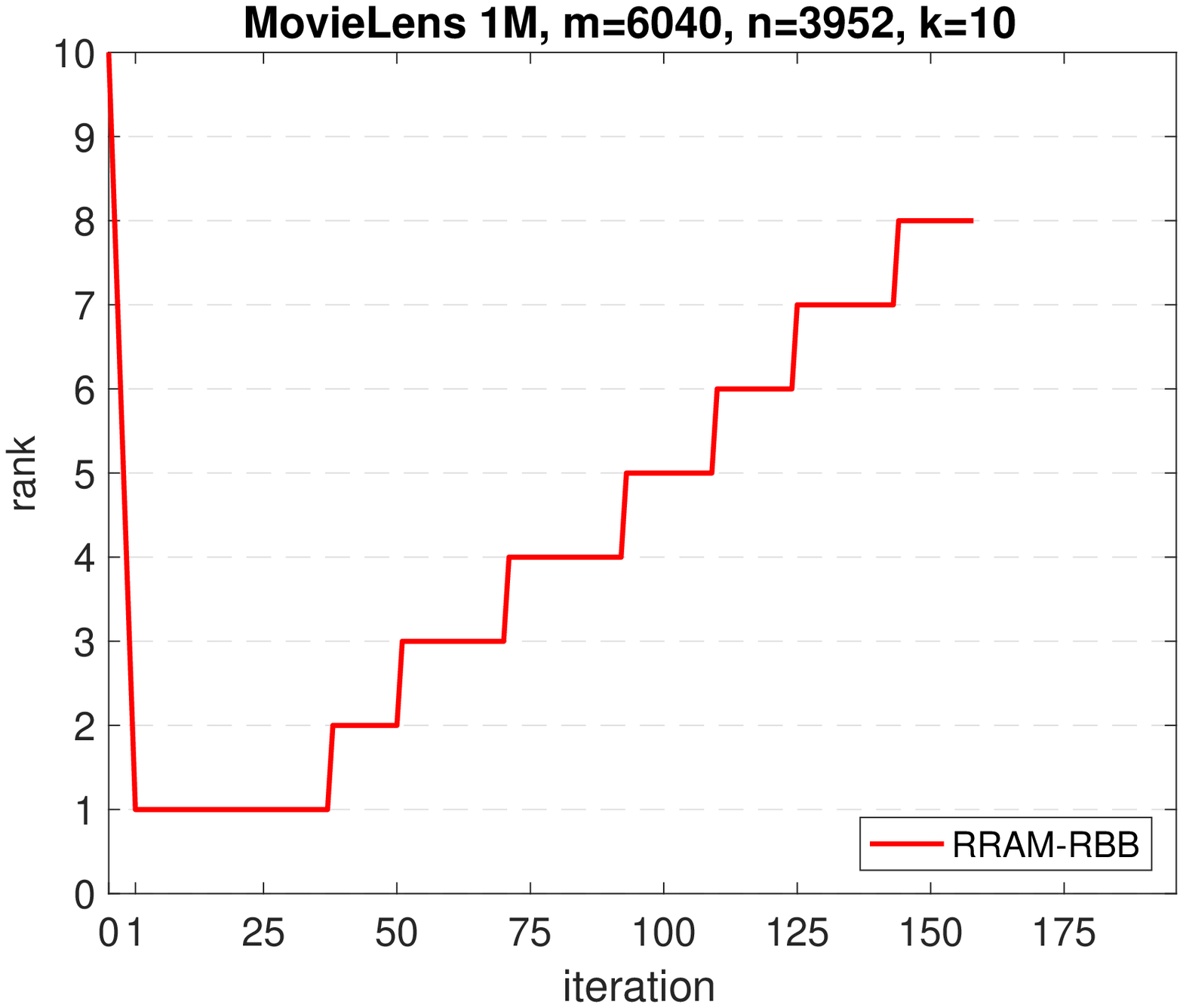}}
	\caption{A comparison on real-world datasets. First row: MovieLens100K. Second row: MovieLens 1M.}
	\label{fig:RRAM-realdata}
\end{figure}

\section{Conclusion}\label{sec:conclusion}
As the set of fixed-rank matrices is not closed, a rank-adaptive mechanism can be a promising way to solve optimization problems with rank constraints. This paper concerns the low-rank matrix completion problem that can be modeled with a bounded-rank constraint. 
A Riemannian rank-adaptive method is proposed, featuring rank increase and decrease mechanisms that are novel in ways discussed in Remark~\ref{remark:tab}.
Numerical comparisons on synthetic and real-world data show that the proposed rank-adaptive method compares favorably with other algorithms in low-rank matrix completion. 
This suggests that the proposed method might also perform well on other low-rank optimization problems, such as those mentioned in~\cite{schneider2015convergence,Zhou2016riemannian,Chi2019overview,Uschmajew2020}.

\begin{acknowledgements}
We would like to thank Bart Vandereycken for helpful discussions on the code ``LRGeomCG" and Shuyu Dong for generously providing his code of the comparison on real-world datasets.
\end{acknowledgements}

%
%



\end{document}